\documentclass[10pt]{amsart}
\usepackage{amsmath}
\usepackage{amsfonts}
\usepackage{amssymb, textcomp}
\usepackage{amsthm}
\usepackage{url}
\usepackage{tikz-cd}
\usepackage{graphicx}
\usepackage{caption}
\usepackage{subcaption}
\usepackage{comment} 
\usepackage{hyperref}
\usepackage{todonotes}
\usepackage{color}
\usepackage{enumerate}
\usepackage{graphicx}

\numberwithin{equation}{section}

\newtheorem{proposition}{Proposition}[section]
\newtheorem{lemma}[proposition]{Lemma}
\newtheorem{theorem}[proposition]{Theorem}
\newtheorem{corollary}[proposition]{Corollary}

\theoremstyle{definition}
\newtheorem{remark}[proposition]{Remark}
\newtheorem{definition}[proposition]{Definition}
\newtheorem{example}[proposition]{Example}

\DeclareMathOperator{\Bl}{Bl}

\DeclareMathOperator{\Aut}{Aut}

\DeclareMathOperator{\Proj}{Proj}

\renewcommand{\L}{\mathcal{L}}
\renewcommand{\phi}{\varphi}

\DeclareMathOperator{\DF}{DF}

\DeclareMathOperator{\Supp}{Supp}

\DeclareMathOperator{\Pic}{Pic}
\DeclareMathOperator{\ord}{ord}
\DeclareMathOperator{\lct}{lct}

\newcommand{\N}{\mathbb{N}}
\newcommand{\R}{\mathbb{R}}
\newcommand{\C}{\mathbb{C}}

\newcommand{\Z}{\mathbb{Z}}
\newcommand{\Q}{\mathbb{Q}}

\newcommand{\pr}{\mathbb{P}}
\renewcommand{\epsilon}{\varepsilon}

\newcommand{\scO}{\mathcal{O}}

\newcommand{\X}{\mathcal{X}}

\newcommand{\F}{\mathcal{F}}

\newcommand{\Y}{\mathcal{Y}}

\DeclareMathOperator{\Vol}{Vol}

\DeclareMathOperator{\wt}{wt}
\newcommand{\E}{\mathcal{E}}

\DeclareMathOperator{\NA}{NA}
\DeclareMathOperator{\INA}{I}
\DeclareMathOperator{\JNA}{J}
\DeclareMathOperator{\HNA}{H}

\def\mX{\mathcal{X}} 
 \def\mL{\mathcal{L}}
\def\ra{\rightarrow }

\newcommand{\barycenter}{\mbox{bar}}

\title[Valuative stability of polarised varieties]{Valuative stability of polarised varieties}

\author[Ruadha\'i Dervan and Eveline Legendre]{Ruadha\'i Dervan and Eveline Legendre}
\address{Ruadha\'i Dervan, DPMMS, Centre for Mathematical Sciences, Wilberforce Road, Cambridge CB3 0WB, United Kingdom}
\email{R.Dervan@dpmms.cam.ac.uk}

\address{Eveline Legendre, Universit\'e Paul Sabatier, Institut de Math\'ematiques de Toulouse, 118 route de Narbonne, France}
\email{ eveline.legendre@math.univ-toulouse.fr }

\begin{document}

\begin{abstract}  

Fujita and Li have given a characterisation of K-stability of a Fano variety in terms of quantities associated to valuations, which has been essential to all recent progress in the area. We introduce a notion of valuative stability for arbitrary polarised varieties, and show that it is equivalent to K-stability with respect to test configurations with integral central fibre. The numerical invariant governing valuative stability is modelled on Fujita's $\beta$-invariant, but includes a term involving the derivative of the volume. We give several examples of valuatively stable and unstable varieties, including the toric case. We also discuss the role that the $\delta$-invariant plays in the study of valuative stability and K-stability of polarised varieties.

\end{abstract}

\maketitle

\section{Introduction}

The  notion of K-stability of a polarised variety (i.e. a projective variety endowed with an ample line bundle) has played a central role in algebraic geometry in recent years. The primary motivation for K-stability is the Yau-Tian-Donaldson conjecture, which states that K-stability should be equivalent to the existence of constant scalar curvature K\"ahler metrics on the polarised variety \cite{yau, tian-inventiones, donaldson-toric}, and also predicts that one should be able to form moduli spaces of K-stable polarised varieties.

While this conjecture is completely open in general, there has been enormous progress on these ideas in the case of Fano varieties. Analytically, it is now known that K-stability is equivalent to the existence of a K\"ahler-Einstein metric on a smooth Fano variety \cite{CDS, tian-inventiones, berman}. Algebraically, the theory has advanced massively, primarily through Fujita and Li's reinterpretation of K-stability in terms of valuations  \cite{fujita-valuative, li}. These ideas, together with significant input from birational geometry, have led to an almost-complete understanding of K-stability of Fano varieties. This is true both abstractly, in the sense that one can now construct moduli spaces of K-stable Fano varieties (though properness remains open\footnote{Properness has now been proven by Liu-Xu-Zhuang \cite{LXZ}.}), and concretely, in the sense that one can now give a very thorough understanding of which Fano varieties are actually K-stable. There are many results along these lines, such as \cite{blum-xu, codogni-patakfalvi, fujita-alpha, AZ} to name only a few.  The valuative approach to K-stability of Fano varieties has been essential to all of these developments.

From this perspective, one of the main issues in understanding K-stability of an arbitrary polarised variety is that we do not yet understand the role played by valuations. The original definition of K-stability, due to Donaldson and building on work Tian, involves \emph{test configurations}: these are $\C^*$-degenerations of the polarised variety $(X,L)$ to another polarised scheme (called the central fibre). Donaldson then assigns a numerical invariant to a test configuration, now called the Donaldson-Futaki invariant, and K-stability means that this invariant is always positive. Fujita and Li reinterpret K-stability by replacing test configurations with valuations on $X$, and replacing the Donaldson-Futaki invariant with numerical invariants associated to the volume and log discrepancy of the valuation \cite{fujita-valuative, li}. 

Our main result gives a complete understanding of how valuations can be used to study K-stability of polarised varieties, primarily based on the ideas of Fujita \cite{fujita-valuative}. We briefly give the definition, before stating the main results. Let $(X,L)$ be an  $n$-dimensional polarised variety and let $F$ be a prime divisor over $X$. Denote by $A_X(F)$ the log discrepancy of $F$, and $\Vol$ the volume function. We define the $\beta$\emph{-invariant} of $F$ by $$\beta(F) = A_X(F)\Vol(L) + n \mu \int_{0}^{\infty} \Vol(L-xF)dx + \int_0^{\infty}{\Vol}(L-xF)'\cdot K_X dx,$$ where ${\Vol}(L-xF)'\cdot K_X $ denotes the \emph{derivative} of the volume in the direction $K_X$, and $$\mu = \mu(X,L) = \frac{-K_X.L^{n-1}}{L^n}$$ is a topological constant. In comparison with Fujita's invariant, the main novelty is the appearance of the derivative of the volume; we note that in general the volume is a continuously differentiable function \cite{BFJ}. As in Fujita's work, an important class of divisorial valuations are those that are \emph{dreamy}; this is a finite generation hypothesis. We then say that $(X,L)$ is \emph{valuatively stable} if $\beta(F)>0$ for all dreamy divisorial valuations $F$. Our main result demonstrates the relationship between valuative stability and K-stability.

\begin{theorem}\label{thm:intromainthm} A polarised variety is valuatively stable if and only it is K-stable with respect to test configurations with irreducible central fibre.  \end{theorem}

This fully explains the role played by valuations in the study of K-stability of polarised varieties. We also prove analogous results for K-semistability, (equivariant) K-polystability and uniform K-stability.  Our proof is modelled on that of Fujita \cite{fujita-hyperplane, fujita-valuative}, and the primary differences arise from the fact that the Donaldson-Futaki invariant takes a significantly simpler form in the Fano setting; this explains the appearance of the derivative of the volume in the $\beta$-invariant. 

In general one should \emph{not} expect that valuative stability is equivalent to K-stability, and equivariant versions of this statement fail in the toric setting. Nevertheless, a deep result of Li-Xu states that for K-stability of \emph{Fano} varieties $(X,-K_X)$, it \emph{is} equivalent to check K-stability with respect to test configurations with irreducible central fibre \cite{li-xu}. Thus Fujita's work implies that valuative stability of Fano varieties is equivalent to K-stability. This therefore explains, from the point of view of valuations, the difference between the Fano theory and the general theory. Moreover, test configurations with smooth, hence irreducible, central fibre play an important role in many analytic works concerning the existence of constant scalar curvature K\"ahler metrics \cite{chen-sun, szekelyhidi}, and hence one should expect Theorem \ref{thm:intromainthm} to be a powerful tool. Theorem \ref{thm:intromainthm} in any case produces a concrete obstruction to K-stability of polarised varieties, which we expect to play a similar role to Ross-Thomas' slope stability \cite{ross-thomas}, in that the resulting criterion should be practically checkable in concrete examples.

Beyond the case of anticanonically polarised Fano varieties, Delcroix has recently shown that there are certain polarised spherical varieties $(X,L)$ for which the existence of a constant scalar curvature K\"ahler metric is equivalent to equivariant K-polystability with respect to test configurations with irreducible central fibre \cite[Section 11]{TD}, and is hence also equivalent to K-polystability \cite{BDL}. Thus we obtain a full valuative criterion for  K-polystability also in this case. Delcroix's examples have polarisations close to the anticanonical class $-K_X$ in their ample cone, and he suggests that an analogous statement should be true for arbitrary polarised spherical varieties, provided the polarisation is sufficiently close to the anticanonical class. Going even further, it seems reasonable to suggest that for general polarisations of Fano varieties $(X,L)$, provided the polarisation is sufficiently close to $-K_X$, it may be the case that K-polystability is equivalent to K-polystability with respect to test configurations with irreducible central fibre, generalising the result of Li-Xu \cite{li-xu} and meaning our valuative criterion would characterise K-polystability for more general polarisations of Fano varieties. 

We hope that some of the numerous applications that the valuative approach to K-stability of Fano varieties can be applied to general polarised varieties through Theorem \ref{thm:intromainthm}, and we plan to return to this in future work. In the present work, we prove some foundational results along these lines. For example, we show that Calabi-Yau varieties and canonically polarised varieties are uniformly valuatively stable provided they have mild singularities. This follows from Theorem \ref{thm:intromainthm} and work of Odaka, but demonstrates how one should use the $\beta$-invariant for general polarised varieties. We also prove alpha invariant bounds modelled on work of Fujita-Odaka \cite{fujita-odaka}.

We also give a complete geometric description in the toric case. For this we take $(X,L)$ to be a $\Q$-factorial polarised toric variety. We then call $F$ a {\it toric prime divisor over} $X$ if there is a normal compact toric variety $Y$ and $\psi: Y\ra X$ a proper birational toric morphism whose exceptional set coincides with $F$, a toric prime divisor on $Y$. By considering only valuations emanating from toric prime divisors we obtain a weak notion of equivariant valuative semistability, which turns out to be equivalent to the classical Futaki invariant~\cite{futaki}, which is a function on holomorphic vector fields on $X$. We then say that the Futaki invariant  vanishes identically if it vanishes for each holomorphic vector field.

\begin{theorem}\label{prop:toricVal=Fut} The Futaki invariant of $(X,L)$ vanishes identically (on the torus) if and only if $$\beta(F) \geq 0$$ for any toric prime divisor $F$ over $X$. 
\end{theorem}

This extends Fujita's result on toric divisorial stability of Fano varieties \cite{fujita-plms}. The proof uses the expression of the classical Futaki invariant of $(X,L)$ as the difference of the barycentres of the the moment polytope and its boundary exhibited in Donaldson's work \cite{donaldson-toric}.  

Beyond the work of Fujita and Li, perhaps the most important foundational development in the study of valuative stability of Fano varieties has been Fujita-Odaka's introduction of the $\delta$-invariant $\delta(L)$ \cite{fujita-odaka}, proved by Blum-Jonsson to equal $$\delta(L) = \inf_F \frac{A_X(F)\Vol(L)}{\int_0^{\infty} \Vol(L-xF) dx}, $$ where the infimum is taken over all prime divisors $F$ over $X$ \cite{blum-jonsson}. It follows from work of Fujita-Odaka and Blum-Jonsson that $\delta(-K_X)\geq 1$ if and only if $(X,-K_X)$ is K-semistable, with $\delta(-K_X)>1$ characterising uniform K-stability. While it is clear from our definition of $\beta(F)$ for general polarised varieties that $\delta(L)$ plays an important role, it is natural to ask whether or not a condition on $\delta(L)$ actually characterises valautive stability more generally. While this seems unlikely, we show that one can provide \emph{sufficient} conditions for valuative stability in terms of $\delta(L)$:

\begin{theorem}\label{thm:introdelta} Write $\delta(L) - \mu(L) =(n+1)\gamma(L)$, and suppose that the line bundle $$(\mu(L) + \gamma(L)) L + K_X$$ is effective. Then $(X,L)$ is uniformly valuatively stable.

\end{theorem}

Note that this recovers one direction of Fujita-Odaka's work as a special case: when $L=-K_X$ and $\delta(-K_X)>1$, $\gamma$ is strictly positive and hence $(\mu + \gamma) L + K_X = \gamma L$ is indeed effective. The hypothesis is reminiscent of the sufficient criterion for uniform K-stability of general polarisations of Fano varieties due to the first author \cite{alpha} (and \cite[Theorem 1.9]{twisted}); the effectivity hypothesis, however, has a slightly different flavour. It is interesting to ask whether one can prove that under the hypothesis of Theorem \ref{thm:introdelta} that $(X,L)$ is actually uniformly K-stable, or admits a constant scalar curvature K\"ahler metric. We refer work of K. Zhang \cite[Corollary 6.13]{zhang} for a closely related analytic result, and also note that, under the alpha invariant hypotheses just mentioned, the existence of such a metric is now known \cite{mabuchi, chen-cheng}. 

To this end, we remark that K. Zhang has recently introduced an \emph{analytic} counterpart $\delta^A(L)$ to the $\delta$-invariant, and has made the strong conjecture that the two invariants agree \cite{zhang}\footnote{K. Zhang has recently established his conjecture \cite{KZ2}.}. Roughly speaking, $\delta^A(L)$ is the optimal constant for which the entropy term in the Mabuchi functional on the space of K\"ahler metrics dominates the $(I-J)$-functional. Thus one can ask whether, under the hypotheses of Theorem \ref{thm:introdelta} but replacing the bound on $\delta(L)$ with one on $\delta^A(L)$, that $(X,L)$ admits a cscK metric; this seems to require new ideas in comparison with the corresponding result concerning the alpha invariant \cite{mabuchi}, and perhaps suggests that there are important properties of $\delta(L)$ yet to be discovered. 

We finally remark that K. Zhang's invariant has another more direct algebro-geometric invariant, defined as the optimal constant for which the discrepancy term of the Donaldson-Futaki invariant (defined in Equation \eqref{HNA}) dominates the minimum norm: \begin{equation}\delta^H(L) = \inf_{(\X,\L)} \frac{\HNA(\X,\L)\Vol(L)}{\|(\X,\L)\|_m},\end{equation} with the infimum taken over all test configurations.  It would follow from Zhang's conjecture and other conjectures surrounding the Yau-Tian-Donaldson conjecture that all three $\delta$-invariants agree, and it is again natural to ask whether one can prove that under the hypotheses of Theorem \ref{thm:introdelta} but replacing the bound on $\delta(L)$ with one on $\delta^A(L)$, $(X,L)$ is uniformly K-stable; this would give some evidence for K. Zhang's conjecture, but also seems to require new ideas.

\subsection*{Outline} In Section \ref{sec:prelims} define the various notions of stability relevant to us, most centrally valuative stability and K-stability. In Section \ref{sec:proofs} we prove Theorem \ref{thm:intromainthm} and variants for other notions of stability, such as K-polystability. In particular \ref{cor:onedirection} proves valuative semistability implies K-semistability with respect to test configurations with integral central fibre, with Lemma \ref{lemma:norm} relating the norms and demonstrating that uniform valuative stability implies uniform K-stability with respect to the same class of test configurations. Proposition \ref{prop:converse} and Lemma \ref{lem:converse} give the converse. Section \ref{sec:examples} provides various examples, including a proof of Theorem \ref{thm:introdelta}, while Section \ref{sec:toric} considers the toric setting, including a proof of Theorem \ref{prop:toricVal=Fut}.

\subsection*{Acknowledgements:} We thank Ivan Cheltsov, Thibaut Delcroix, Kento Fujita and Yaxiong Liu for helpful comments, and the referee for their suggestions.

\subsection*{Notation:} We work throughout over the complex numbers, though  everything goes through over an algebraically closed field of characteristic zero. All varieties are taken to be normal.

\section{Preliminaries}\label{sec:prelims}
\subsection{Valuations and associated invariants} We define the invariants associated to valuations of interest to us, and refer to \cite{Lazarsfeld1} or the work of Fujita for an introduction.

Let $X$ and $Y$ be normal projective variety, and let $\pi: Y \to X$ be a surjective birational morphism, with $X$ and $Y$ of dimension $n$. 

\begin{definition} A prime divisor $F \subset Y$ for some $Y$ is called a \emph{prime divisor over} $X$.\end{definition}

We view $F$ as defining a divisorial valuation on $X$; in particular, the information associated to $F$ which we will be concerned with depends only on the valuation associated to $F$. In particular, one can always take $Y$ to be smooth by passing to a resolution of singularities.

Define a vector subspace $H^0(X, kL-xF)  \subset H^0(X, kL)$ via the identifications  $$H^0(X, kL-xF) = H^0(Y,k\pi^*L - xF) \subset H^0(Y,k\pi^*L) \cong H^0(X,kL),$$ where we note that the last isomorphism is canonical.

\begin{definition} For $x \in \R$, we define the \emph{volume} of $L-xF$ to be $$\Vol(L-xF) := \limsup_{k \to \infty}\frac{\dim H^0(X,kL-\left \lfloor{kx}\right \rfloor 
F)}{k^n/n!}.$$\end{definition}

A basic property is that the limsup defining the volume is actually a limit \cite[Remark 2.2.50]{Lazarsfeld1}. The volume function extends by homogeneity from $\Pic(X)$ to $\Pic_{\Q}(X)$, and continuously from $\Pic_{\Q}(X)$ to $\Pic_{\R}(X)$. The \emph{big cone} of $X$ is the locus inside $\Pic_{\R}(X)$ of $\R$-line bundles with positive volume; this is an open cone. 

\begin{theorem}\cite{BFJ} The volume is a continuously differentiable function on the big cone of $X$. \end{theorem}

For a big line bundle $L$ on $X$, and another arbitrary line bundle $H$ on $X$, we denote by $\Vol(L)'\cdot H$ the value $$\Vol(L)'\cdot H = \frac{d}{dt}\Vol(L +tH)\Big|_{t=0}.$$

\begin{definition} We define the \emph{pseudoeffective threshold} of $F$ with respect to $L$ to be $$\tau_L(F) = \sup \{x \in \R | \Vol(L-xF)>0.\}$$\end{definition}  Observe that if $j>\tau_L(F)$, then  $H^0(Y, m(\pi^*L -jF)) =0$.

We now assume that $K_X$ and $K_Y$ are $\Q$-Cartier. 

\begin{definition} The \emph{log discrepancy} $A_X(F)$ is defined to be $$\ord_F(K_Y - \pi^*K_X) + 1.$$\end{definition}

In all cases of interest to us, $X$ will have either log canonical or log terminal singularities, from which it follows that $A_X(F) \geq 0$ in the former case and $A_X(F) > 0$ in the latter.

An important class of divisors are those which are \emph{dreamy}.

\begin{definition}\cite[Definition 1.3]{fujita-valuative} We say that $F$ is \emph{dreamy} if for some (equivalently any) $r \in \Z_{>0}$ the $\Z_{\geq 0}^{\otimes 2}$-graded $\C$-algebra $$\bigoplus_{j,k\in \Z_{\geq 0}} H^0(X,krL-jF)$$ is finitely generated.\end{definition}

While this concept depends on $L$, we will always omit this from our notation.

\begin{example}\label{dreamy-fano} Suppose $Y$ a Fano type variety, in the sense that there exists an effective $\Q$-divisor $D$ on $Y$ such that $(Y,D)$ is log terminal and such that $-(K_Y+D)$ is big and nef. Then any $F \subset Y$ is dreamy \cite[Corollary 1.3.1]{BCHM}. This applies, for example, if $Y$ is itself Fano or toric. 

Not all prime divisors are dreamy, however. An example of a non-dreamy prime divisor $F$ over $(\pr^2,-K_{\pr^2})$ has been produced by Fujita \cite[Example 3.8]{fujita-maeda}.
\end{example}

Now let $(X,L)$ be an $n$-dimensional normal polarised variety, by which we mean that $L$ is an ample line bundle on $X$. The invariant of ultimate interest to us is the following analogue of Fujita's $\beta$-invariant. Denote by $$\mu = \mu(X,L) = \frac{-K_X.L^{n-1}}{L^n}$$ the \emph{slope} of $(X,L)$. It is interesting to note the value $\mu(X,L)$ can be interpreted as a derivative of the volume, namely $$-n\mu(X,L) = \Vol'(L)\cdot K_X.$$

\begin{definition} Let $F$ be a prime divisor over $(X,L)$. We define the $\beta$\emph{-invariant} of $F$ to be $$\beta(F) = A_X(F)\Vol(L) + n \mu \int_{0}^{\infty} \Vol(L-xF)dx + \int_0^{\infty}\Vol'(L-xF)\cdot K_X dx.$$ \end{definition}

\begin{remark} The integrands vanish once $x \geq \tau(F)$, meaning one can instead define $$\beta(F) = A_X(F)\Vol(L) + n \mu \int_{0}^{\tau(F)} \Vol(L-xF)dx + \int_0^{\tau(F)}\Vol'(L-xF)\cdot K_X dx.$$ \end{remark}

\begin{example} Suppose $L=-K_X$, so that $(X,-K_X)$ is a Fano variety. Then we show in Corollary \ref{cor:fuj} using integration by parts that $$\beta(F) = A_X(F)\Vol(-K_X) - \int_{0}^{\tau(F)} \Vol(-K_X-xF)dx,$$ which is precisely Fujita's $\beta$-invariant \cite{fujita-valuative}. Thus our invariant is a generalisation of Fujita's invariant to arbitrary polarised varieties.
\end{example}

There are three natural numerical invariants on the space of divisorial valuations which, roughly speaking, play the role of norms. Following Fujita and Blum-Jonsson \cite{fujita-valuative, blum-jonsson}, we set \begin{align*} S(F) &= \frac{\int_0^{\infty}\Vol(L-xF)dx}{L^n}, \\ j(F) &= L^n(\tau(F) - S(F)) = \int_0^{\infty} (\Vol(L) - \Vol(L-xF))dx.\end{align*}

\begin{proposition}\label{prop:lipshitz-norms} The quantities $\tau(F), S(F)$ and $j(F)$ are strictly positive on the space of non-trivial divisorial valuations. Moreover, they are each mutually uniformly bounded above and below. 
\end{proposition}

That is, for example, there are constants $c_1,c_2>0$ such that for all non-trivial prime divisors $F$ we have $$0 <c_1 S(F) \leq \tau(F) \leq c_2 S(F).$$

\begin{proof} The proof is essentially the same as that of Fujita and Fujita-Odaka in the case of Fano varieties with $L=-K_X$ \cite{fujita-odaka, fujita-proceedings}, and in particular the perspective of viewing these as analogous to norms is due to Fujita \cite{fujita-proceedings}. 

We begin by noting that each is strictly positive on non-trivial divisorial valuations. This is clear for $\tau(F)$ and $S(F)$, and for $j(F)$ follows from the fact that $\Vol(L-xF)<\Vol(F)$ for each $x \in (0,\tau(F))$. 

Thus what remains to show is Lipschitz equivalence. We claim \begin{equation}\label{eqn:lipschitz1}\frac{1}{n+1}\tau(F)\leq S(F) \leq \tau(F).\end{equation}

Arguing as in \cite[Lemma 1.2]{fujita-odaka}, note that $$\int_0^{\infty}\Vol(L-xF)dx \leq L^n\tau(F)$$ since $\Vol(L-xF) \leq \Vol(F)$. Thus $S(F) \leq \tau(F)$. Concavity of the volume function gives $$\Vol(L-xF) \geq L^n\left(\frac{x}{\tau(F)}\right)^n,$$ which implies $S(F) \geq\frac{1}{n+1}\tau(F).$

By transitivity of Lipschitz equivalence, what remains is to show that $\tau(F)$ is Lipschitz equivalent to $j(F)$. We claim $$\frac{1}{n+1}\tau(F)L^n \leq j(F) \leq \frac{n}{n+1} \tau(F) L^n,$$ which is equivalent to asking $$\frac{1}{n+1}\tau(F) \leq\tau(F) - S(F)  \leq \frac{n}{n+1} \tau(F).$$ But this is a simple rearrangement of Equation \eqref{eqn:lipschitz1}.\end{proof}

\begin{remark} These quantities are closely related to the functionals $I, J$ and $I-J$ on the space of K\"ahler potentials in a fixed K\"ahler class, together with their analogues for test configurations, which play similar roles. In particular, $S(F)$ is analogous to the minimum norm of a test configuration; this observation is due to Fujita and Blum-Liu-Zhou \cite{fujita-proceedings, blum-liu-zhou}. We also note that Boucksom-Jonsson have recently proven a stronger result than Proposition \ref{prop:lipshitz-norms}, that applies to more general ``non-Archimedean metrics'' \cite[Theorem C]{BJ}.

 \end{remark}

\begin{definition} We say that a polarised variety $(X,L)$ is
\begin{enumerate}[(i)]
\item \emph{valuatively semistable} if $$\beta(F) \geq 0$$ for all dreamy prime divisors $F$ over $(X,L)$;
\item \emph{valuatively stable} if $$\beta(F) > 0$$ for all non-trivial dreamy prime divisors $F$ over $(X,L)$;
\item \emph{uniformly valuatively stable} if there exists an $\epsilon>0$ such that $$\beta(F) \geq \epsilon j(F)$$ for all dreamy prime divisors $F$.
\end{enumerate}
\end{definition}

\noindent One could, of course, use any of the three Lipschitz equivalent norms; we use $j(F)$ to mirror Fujita's original definitions in the Fano setting.

\begin{remark} The invariants of interest scale as $$\beta_{kL}(F) = k^n \beta_L(F), \qquad j_{kL}(F)= k^{n+1} j_L(F).$$ Thus $(X,L)$ is valuatively semistable, for example, if and only if $(X,kL)$ is so. This allows us to scale $L$ harmlessly in many of the arguments, simplifying the notation. 

\end{remark}

\begin{remark}\label{rmk:strongvsweak} In general, it is not clear whether or not to expect that in the definitions of valuative uniform and semistability one can remove the dreaminess hypothesis; it is unlikely that the analogue of this is true for stability.
 \end{remark}

\subsection{K-stability}\label{sec:kstability} The primary aim of the present work is to relate valuative stability to K-stability. This involves a class of degenerations, called test configurations, as well as an associated numerical invariant.

\begin{definition}\cite{donaldson-toric,tian-inventiones} A \emph{test configuration} is a normal variety $\X$ with a line bundle $\L$ together with
\begin{enumerate}[(i)]
\item a flat projective morphism $\X \to \C$, making $\L$ relatively ample,
\item a $\C^*$-action on $(\X,\L)$ making $\pi$ an equivariant map with respect to the usual action on $\C$,
\end{enumerate} such that $(\X_t,\L_t) \cong (X,L^r)$ for all $t\neq 0$ and for some $r \in \Z_{>0}$. We call $r$ the \emph{exponent} of the test configuration.
\end{definition}

Since there is a $\C^*$-action on the central fibre $(\X_0,\L_0)$, there is an induced $\C^*$-action on $H^0(\X_0,\L_0^k)$ for all $k$. The total weight of this $\C^*$-action is, for $k \gg 0$, a polynomial of degree $n+1$ which we denote $$\wt H^0(\X_0,\L_0^k) = b_0 k^{n+1} + b_1 k^n + O(k^{n-1}).$$ We similarly denote the Hilbert polynomial by $$\dim H^0(\X_0,\L_0^k) = a_0 k^{n} + a_1 k^{n-1} + O(k^{n-1}).$$

\begin{definition}\cite{donaldson-toric} The \emph{Donaldson-Futaki invariant} of $(\X,\L)$ is defined to be $$\DF(\X,\L) = \frac{b_0a_1 - b_1 a_0}{a_0}.$$\end{definition}

\begin{definition}\cite{donaldson-toric} We say that $(X,L)$ is
\begin{enumerate}[(i)]
\item  \emph{K-semistable} if for all test configurations $(\X,\L)$, we have $\DF(\X,\L)\geq 0$;
\item \emph{K-stable} if for all non-trivial test configurations $(\X,\L)$, we have $\DF(\X,\L)> 0$.
\end{enumerate} Here a test configuration is \emph{trivial} if it is of the form $(X \times \C, L)$, with trivial $\C^*$-action on $X$.
\end{definition}

In the equivariant setting, we assume $G \subset \Aut_0(X,L)$ is a connected subgroup of the connected component of the identity $\Aut_0(X,L)$ of $\Aut(X,L).$ A test configuration is then called $G$-\emph{equivariant} if there is a $G$-action on $(\X,\L)$ which commutes with $\pi$ and extends the usual action on the fibres $(\X_t, \L_t)$ over $t \in \setminus \{0\}$.

\begin{definition} We say that $(X,L)$ is $G$\emph{-equivariantly K-polystable} if for all $G$-equivariant test configurations $(\X,\L)$ we have $\DF(\X,\L)\geq 0$, with equality if and only if $(\X_0, \L_0) \cong (X,L^r)$ for some  $r \in \Z_{>0}$. \end{definition}

Another useful perspective on the Donaldson-Futaki invariant is via intersection theory. For this it is convenient to rescale $\L$ so that on the general fibre it is isomorphic to $L$ rather than $L^r$; this may make $\L$ a $\Q$-line bundle. We will always assume that we have performed this scaling.  While this will not be used in our proof of Theorem \ref{thm:intromainthm}, it will be important in motivating our discussion concerning the various delta invariants contained in the introduction, and will give some motivation for the definitions of the norms which will be introduced momentarily. A test configuration can be compactified to a family over $\pr^1$, by compactifying trivially at infinity. We denote this compactification, abusively, by $(\X,\L)$, with $\X$ now a projective variety.

\begin{theorem}\label{thm:intersections}\cite{odaka, wang} We have $$\DF(\X,L) = \frac{n}{n+1}\mu(X,L)\L^{n+1} + \L^{n}.K_{\X/\pr^1},$$ where $K_{\X/\pr^1}$ denotes the relative canonical class, which is a Weil divisor by normality of $\X$. Moreover, $$b_0 = \frac{\L^{n+1}}{(n+1)!}.$$

\end{theorem}

We next turn to the analogues of norms for test configuration. Much as for valuations, there are three natural norm-type quantities one can use. Passing to a resolution of indeterminacy of the rational map $X\times\pr^1 \dashrightarrow \X$ if necessary, we may assume that $\X$ admits a morphism to $X \times \pr^1$. All quantities defined in this way are independent of choice of resolution of indeterminacy.

\begin{definition}\cite{twisted, BHJ} We define \begin{align*}\JNA^{\NA}(\X,\L) &= \frac{\L.L^n}{L^n} - \frac{\L^{n+1}}{(n+1)L^n}, \\ \|(\X,\L) \|_m &= \frac{\L^{n+1}}{(n+1)L^n} - \frac{\L^n.(\L-L)}{L^n}, \\ \INA^{\NA}(\X,\L) &= \|(\X,\L)\|_m + \JNA^{\NA}(\X,\L),\end{align*} and call $ \|(\X,\L)\|_m$ the \emph{minimum norm} of the test configuration.
\end{definition}

We have divided the minimum norm by an unimportant factor of $L^n$ in comparison with it original definition \cite{twisted}, to remain consistent with the literature elsewhere. Boucksom-Hisamoto-Jonsson denote $$ \INA^{\NA}(\X,\L) - \JNA^{\NA}(\X,\L) = \|(\X,\L) \|_m,$$ to emphasise the links to the associated functionals used in K\"ahler geometry; we prefer to use the terminology of \cite{twisted} to emphasise that it plays the role of a norm (which can more geometrically be defined via the minimum weight of an associated $\C^*$-action, explaining its name).

\begin{proposition}\cite{twisted, BHJ} The quantities $\JNA^{\NA}(\X,\L), \INA^{\NA}(\X,\L)$ and $\|(\X,\L)\|_m$ are strictly positive when $(\X,\L)$ is non-trivial, and are moreover are mutually uniformly bounded below and above. \end{proposition}

The mutual uniform boundedness is due to Boucksom-Hisamoto-Jonsson.

\begin{definition}\cite{twisted, BHJ} We say that $(X,L)$ is \emph{uniformly K-stable} if there exists and $\epsilon>0$ such that for all test configurations $(\X,L)$ $$\DF(\X,\L) \geq \epsilon \|(\X,\L)\|_m.$$\end{definition}

The final numerical invariant which plays only a minor role in the present work is the \emph{non-Archimedean entropy} \cite[Definition 7.17]{BHJ}: \begin{equation}\label{HNA}\HNA(\X,\L) = \frac{ \L^n.K_{\X/X\times\pr^1}}{L^n} + \frac{\L^n.(\X_{0,red} - \X_0)}{L^n},\end{equation}  computed as above on a resolution of indeterminacy, and with $\X_{0,red}$ denoting the the central fibre given the induced reduced structure. This invariant is, up to the error term vanishing when the central fibre of the test configuration is reduced, the same as the ``discrepancy term'' of Odaka \cite{odaka}.

By work of Witt Nystr\"om, a test configuration induces a filtration of the coordinate ring of $(X,L)$ \cite{witt-nystrom}.

\begin{definition} A \emph{filtration} of $$R = \oplus_{k \geq 0} H^0(X,kL)$$ is a chain of vector subspaces$$R = F_0R \supset \cdots \supset F_i R \supset F_{i+1} R \supset \cdots \supset \C$$
which is
\begin{enumerate}[(i)]
\item \emph{multiplicative}, in the sense that $(F_i R_l) (F_j R_m) \subset F_{i+j} R_{l+m}$;
\item \emph{homogeneous}, in the sense that if $f\in F_iR$ then each homogeneous piece of $f$ is in $F_iR$.
\end{enumerate} 
\end{definition}

\begin{theorem}\label{wittnystrom}\cite{witt-nystrom} Let $(\X,\L)$ be a test configuration, and denote $$ \F^jR_k = \{ s \in R_k | t^{ {-j} } s \textrm{ is holomorphic on }\X\}.$$  Then $\F$ is a filtration.

Set $$\lambda^{(k)}_{\min} = \inf\{j \in \R | \F^j V_k \neq V_k\}, \qquad \lambda^{(k)}_{\max} = \sup\{j \in \R | \F^j V_k \neq 0\},$$ and let $$\lambda_{\min} = \lim_{k \to \infty} \frac{\lambda_{\min}^{(k)}}{k}, \qquad  \lambda_{\max} = \lim_{k \to \infty} \frac{\lambda_{\max}^{(k)}}{k}.$$ Then the weight polynomial of $(\X,\L)$ is given for $k \gg 0$ by \begin{align*}w(k) &=\sum_{j = \lambda^{(k)}_{\min}}^{\lambda^{(k)}_{\max}} j (\dim \F^j V^k - \dim \F^{j+1} V^k), \\ &= \sum_{j=\lambda_{\min}^{(k)}}^{\lambda^{(k)}_{\max}} \dim \F^j V_k + \lambda_{\min}^{(k)} \dim V_k. \end{align*}

\begin{remark}\label{rmk:maxweight} In fact, by rescaling $L$ so that the test configuration has exponent one, one can assume that $\lambda_{\max}^{(k)} = k\lambda_{\max}$ and $\lambda_{\min}^{(k)} = k\lambda_{\min}$. Geometrically, taking $s \in H^0(\X_0,\L_0)$ is a section of maximal weight, for example, then $s^{\otimes k}$ is a section of $H^0(\X_0,k\L_0)$ of weight $k\wt(s)$. It follows from \cite[Lemma 4]{phong-sturm} that no section can have weight greater than that of $s^{\otimes k}$.
\end{remark}

\end{theorem}

Of most importance for us will be \emph{integral test configurations}. Recall that a scheme is integral if it is reduced and irreducible.

\begin{definition} We say that a test configuration $(\X,\L)$ is \emph{integral} if its central fibre $\X_0$ is an integral scheme. We then say that a polarised variety $(X,L)$ is
\begin{enumerate}[(i)]
\item \emph{integrally K-semistable} if for all integral test configurations $(\X,\L)$ we have $$\DF(\X,\L) \geq 0;$$
\item  \emph{integrally K-stable} if for all non-trivial integral test configurations $(\X,\L)$ we have $$\DF(\X,\L) > 0;$$
\item \emph{uniformly integrally K-stable} if there exists an $\epsilon>0$ such that for all integral test configurations $(\X,\L)$ we have $$\DF(\X,\L) \geq \epsilon \|(\X,\L)\|_m;$$
\item $G$-\emph{equivariantly integrally K-polystable}  if for all $G$-equivariant integral test configurations $(\X,\L)$ we have $$\DF(\X,\L) \geq 0,$$ with equality only when $(\X,\L)$ is a product test configuration. When $G$ is taken to be the identity, we simply say $(X,L)$ is \emph{integrally K-polystable}.
\end{enumerate}
 \end{definition}
 
\noindent  The total space $\X$ of a test configuration is always an irreducible variety, hence $\X$ itself is always an integral scheme, so no confusion will arise from the terminology.

\begin{example} Integral K-semistability should not be equivalent to K-semistability. Counterexamples to the equivariant version of this claim arise in the toric setting, giving strong evidence that this claim fails in general. Restricting to toric test configurations (that is, those equivariant under the torus action), the only integral toric test configurations are those induced by $\C^*$-actions on the variety itself. Thus a counterexample to the equivariant version of this question is given by any toric variety which has vanishing Futaki invariant, but which is K-unstable. Examples of this kind have been produced by Donaldson \cite[Section 7.2]{donaldson-toric}. \end{example}

The key hypothesis is the \emph{irreducibility} of the central fibre, rather than it being reduced:

\begin{lemma} Integral K-semistability is equivalent to asking $$\DF(\X,\L) \geq 0$$ for all test configurations $(\X,\L)$ with irreducible central fibre. Analogous statements also hold for (equivariant) K-polystability and uniform K-stability. \end{lemma}

\begin{proof} The proof follows the proofs of analogous statements for K-stability of polarised varieties \cite{ADL,li-xu} \cite[Proposition 7.15]{BHJ}. Suppose $(\X,\L)$ is a test configuration with central fibre irreducible, but not reduced. Taking the normalised base change $(\X_{(d)}, \L_{(d}))$ over the finite cover $\C \to \C$ induced by $t \to t^d$ for $d$ sufficiently divisible induces a test configuration with reduced and irreducible central fibre. But from Boucksom-Hisamoto-Jonsson  \cite[Proposition 7.8 and Proposition 7.15]{BHJ}, we have \begin{align*}\DF(\X_{(d)}, \L_{(d)}) &\geq d\DF(\X, \L), \\  \|(\X_{(d)}, \L_{(d)})\|_m &= d\|(\X, \L)\|_m,\end{align*} with equality in the first equation if and only if $(\X_0,\L_0)$ is actually reduced. This proves the result.
\end{proof}

This justifies why in the Introduction we only mention test configurations with irreducible central fibre.

\section{Integral K-stability and valuative stability}\label{sec:proofs}

Here we prove our main result:

\begin{theorem} A polarised variety $(X,L)$ is 
\begin{enumerate}[(i)]
\item valuatively semistable if and only if it is integrally K-semistable;
\item valuatively stable if and only if it is integrally K-stable;
\item uniformly valuatively stable if and only if it is uniformly integrally K-stable.
\end{enumerate}
\end{theorem}

The proof will take the entirety of the current section. Many of the ideas involved in our proof of the above follow the work of Fujita in the case of Fano varieties \cite{fujita-valuative, fujita-hyperplane}, and the primary differences arise due to the difference in the definition of the Donaldson-Futaki invariant for arbitrary polarised varieties. In particular, for integral test configurations for Fano varieties, one can understand the Donaldson-Futaki invariant entirely from the leading order term $b_0$ of the associated weight polynomial, while this is no longer the case for arbitrary polarised varieties.

\subsection{Passing from a valuation to a test configuration}

We begin by showing that K-semistability, and its variants respectively, implies valuative semistability, and its variants. Let $F \subset Y \to X$ be a dreamy prime divisor over $X$. The goal of the present section is to produce an integral test configuration associated to $F$. The approach follows the streamlined strategy of Fujita, improving on his earlier technique.

For any $k \in \Z_{>0}$, denote the vector space $$V_k = H^0(X, kL).$$ Associate to the prime divisor $F$ over $X$ the vector subspace $$\F^j V_k = H^0(X, kL- j  F) \text{ if } j \geq 0,$$ with $\F^j V_k = V_k$ otherwise. Since $F$ is dreamy, we can define a scheme $$\X = \Proj_{\C} \bigoplus_{k \in \Z_{\geq 0}} \left( \bigoplus_{j \in \Z} t^{-j} \F^j V_k\right)$$ by taking relative $\Proj$ over $\C$. Thus by the relative $\Proj$ construction $\X$ admits a morphism to $\C$ and a line bundle $\L= \scO(1)$ which is relatively ample over $\C$.

\begin{lemma}\label{associated-test-config} $(\X,\L)$ is an integral test configuration for $(X,L)$.
\end{lemma}

\begin{proof} The proof is in essence identical to Fujita's proof in the Fano setting \cite[Lemma 3.8]{fujita-hyperplane}, but we recall his proof for the reader's convenience. It follows from \cite[Proposition 2.15]{BHJ} and \cite{witt-nystrom, szekelyhidi-filtrations} that $(\X,\L)$ is a test configuration (apart from the claim that $\X$ is normal, which we will shortly demonstrate). 

We show that $\X_0$ is an integral scheme, by showing that its coordinate ring is an integral domain. Note that by construction its coordinate ring is given by $$\bigoplus_{k \geq 0} H^0(\X_0, k\L_0) \cong \bigoplus_{j,k \in \Z_{\geq 0}} \F^j V_k/\F^{j+1}V_k,$$ and denote $S_{k,j} = \F^j V_k/\F^{j+1}V_k.$ Suppose $s_1 \in S_{k_1,j_1}\backslash \{0\}, s_2 \in S_{k_2,j_2}\backslash \{0\}$. Thus by definition  of $\F^j V_k$, the product section $s_1 \otimes s_2$ vanishes precisely $j_1+j_2$ times along $F$. It follows that $s_1 \otimes s_2 \in S_{k_1+k_2,j_1+j_2}\backslash\{0\},$ which shows that $\X_0$ is integral. Since $\X_0$ is reduced, it follows that $\X$ is normal. \end{proof}

We next interpret the Donaldson-Futaki invariant of $(\X,\L)$ in terms of the filtration $\F$ associated to $F$. By construction of $(\X,\L)$ and Theorem \ref{wittnystrom}, it follows that the associated weight polynomial takes the form \begin{align} \begin{split}\label{eqn:weightpoly}w(k)&= \sum_{j=\lambda_{\min}^{(k)}}^{\lambda^{(k)}_{\max}} \dim \F^j V_k + \lambda_{\min}^{(k)} \dim V_k, \\ &= f(k) +\lambda_{\min}^{(k)} \dim V_k,\end{split}\end{align} where\begin{align}
   \begin{split}\label{eqn:f}
 f(k) &= \sum_{j=\lambda_{\min}^{(k)}}^{\lambda^{(k)}_{\max}} \dim \F^j V_k, \\
  &=f_{n+1} k^{n+1} + f_nk^n + O(k^{n-1}). 
   \end{split}
\end{align}As usual, the weight polyonial is only a genuine polynomial for $k \gg 0$. Rescaling $L$ so that the test configuration has exponent one, by Remark \ref{rmk:maxweight} we may assume that $\lambda_{\min}^{(k)} = k\lambda_{\min}, \lambda_{\max}^{(k)} = k\lambda_{\max}.$

\begin{corollary}\label{cor:testconfig-invariants} The numerical invariants of the test configuration are given by $$\DF(\mX,\mL) f_n + n\frac{\mu(X,L)}{2}f_{n+1}  $$ and $$ \lambda_{\max} =\tau(F), \quad \lambda_{\min} = 0.$$
\end{corollary}

\begin{proof} The description of $\lambda_{\min}, \lambda_{\max}$ follows immediately from the description of the filtration. For $k \gg 0$  asymptotic Riemann-Roch provides $$\dim H^0(X,kL) =  \frac{L^n}{n!}k^n + \frac{-K_X.L^{n-1}}{2(n-1)!}k^{n-1} + O(k^{n-2}),$$ from which the result follows using  Theorem \ref{wittnystrom}.
\end{proof}

\begin{remark} We will shortly see that the pseudo-effective threshold $\tau(F)$ is rational, since $F$ is assumed to be dreamy. \end{remark}

\subsection{Running the MMP}\label{sec:MMP}

In the range $x < \tau(F)$ we are interested in, the  line bundle $L-xF$ is merely a big line bundle in general. We next use the minimal model program to produce birational models of $Y$ on which $L-xF$ is actually ample. The ampleness will then allow us to give a geometric understanding of the Donaldson-Futaki invariant of the test configuration produced in the previous section. The following is a direct consequence of \cite{KKL}, using the dreaminess hypothesis.

\begin{theorem}\label{thm:KKL}\cite[Theorem 4.2]{KKL} There exists a sequence of rational numbers $$ 0 = \tau_0 < \tau_1 < \cdots < \tau_m = \tau_L(F),$$ and birational contractions $$\phi_j: Y \dashrightarrow Y_j$$ such that $(\phi_j)_*(L-xF)$ is ample for all $x \in (\tau_{j-1},\tau_j)$ and semiample for all $x \in [\tau_{j-1},\tau_j]$, and each $Y_j$ is a normal projective variety. 

Moreover, for $x \in (\tau_{j-1},\tau_j)$, the map $\phi_j$ is $L-xF$ negative, in the sense that letting $(p,q): Z \to Y \times Y_j$ be a resolution of indeterminacy we have $$p^*(L-xF) = q^*((\phi_j)_*(L-xF)) + E,$$ where $E \geq 0$ is effective and $\Supp E$ contains the proper transform of the $\phi_j$-exceptional divisors. In particular, for all $k \geq 0$ and $x \in (\tau_{j-1},\tau_j)$ there is a canonical isomorphism \begin{equation}\label{eqn:KKL}H^0(Y,k(L-xF)) = H^0(Y_j,  (\phi_j)_*(kL-kxF)).\end{equation}
\end{theorem}

Denote $$L_j = (\phi_j)_* L, \qquad F_j = (\phi_j)_* F.$$ The most important consequence is the following.

\begin{proposition}\label{vol-1} The leading coefficients of $f(k)$ are given as \begin{align*}f_{n+1} &= \sum_{j=1}^m \frac{1}{n!}\int_{\tau_{j-1}}^{\tau_j} (L_j-xF_j)^ndx, \\ f_n &= -\sum_{j=1}^m \frac{1}{2(n-1)!}\int_{\tau_{j-1}}^{\tau_j}((L_j-xF_j)^{n-1}.(K_{Y_j}+F_j)dx.\end{align*} Moreover for $x \in [\tau_{j-1},\tau_j]$, we have an equality $$\Vol(L-xF) = (L_j-xF_j)^n.$$
\end{proposition}

\begin{proof} The claim concerning the coefficients of $f(k)$ follows from Fujita's variant of asymptotic Riemann-Roch \cite[Proposition 4.1]{fujita-plms}, using the isomorphism of Equation \eqref{eqn:KKL}. This isomorphism also proves the claim concerning the volume, using as well that the volume function is continuous in $x$.\end{proof}

It remains to interpret $f_n$ more geometrically. We will require the following minor variant of \cite[Claim 5.6]{fujita-valuative}.

\begin{lemma}\label{discrepancy}
Let $F' \neq F$ be a $\pi$-exceptional divisor. Then $F'$ is $\phi_j$-exceptional. Thus $$ K_{Y_j}-( \phi_{j})_*\pi^*K_X = (A_X(F)-1)F_j.$$
\end{lemma}

\begin{proof}

The claim concerning $\phi_j$-exceptionality follows from \cite[Claim 5.6]{fujita-valuative}, whose proof does not use anything specific to his situation that does not apply to ours. That all such $F'$ are $\phi_j$-exceptional implies $$ (\phi_j)_*(K_{Y}-\pi^*K_X) = (A_X(F)-1)F_j.$$ But since $\phi_j$ is a birational contraction, $$(\phi_j)_*(K_{Y}) = K_{Y_j},$$ proving the final statement. \end{proof}

The following geometric description of $f_n$ explains the appearance of the derivative of the volume in $\beta(F)$.

\begin{lemma}\label{vol-2} For any Cartier divisor $E$ on $Y$ and any $x \in [\tau_{j-1},\tau_j]$, there is an equality $$(L_j-xF_j)^{n-1}.((\phi_{j})_*E) =\frac{1}{n} \Vol'(L-xF) \cdot E.$$ 

\end{lemma}

\begin{proof}

We argue analogously to the above. By differentiability of the volume, it is enough to prove the result for any fixed $x \in (\tau_{j-1},\tau_j)$. Thus $\phi_j$ is $L-xF$ negative,  and hence it is also $L-xF+t\pi^*E$ negative for all $t$ sufficiently small. It then follows from \cite[Remark 2.4 (i)]{KKL} that for all $m \geq 0$ for which $L-xF+tE$ is integral that $$H^0(Y,m(L-xF+tE)) = H^0(Y_j, m(L_j-xF_j+t(\phi_j)_*E)),$$ which implies $$\Vol(L-xF+t\pi^*E) = (L_j-xF_j+t (\phi_j)_*E)^{n}$$ as for $t$ small this line bundle is ample. Differentiating gives the result. \end{proof} 

We are now in a position to relate the numerical invariants of interest.

\begin{proposition}\label{prop:relation-between-DF} The Donaldson-Futaki invariant of $\DF(\mX,\mL)$ is given by \begin{align*}2(n-1)!DF(\mX,\mL) &= A_X(F)\Vol(L) + n\mu\int_0^{\tau(F)}\Vol(L-xF)dx + \int_0^{\tau(F)}\Vol'(L-xF)\cdot K_X, \\ &=\beta(F).\end{align*}
\end{proposition}

\begin{proof}
By what we have proven so far and Lemma \ref{discrepancy}

$$2(n-1)! \DF(\mX,\mL) =  \sum_{j=1}^m\int_{\tau_{j-1}}^{\tau_j}(L_j-xF_j)^{n-1}. \left( ((\phi_{j})_*\pi^*K_X + A_X(F)F_j)  +  (L_j-xF_j)\right)dx.$$ Thus by Proposition \ref{vol-1} and Lemma \ref{vol-2} $$2(n-1)! \DF(\mX,\mL)  = \int_0^{\tau(F)} (\mu\Vol(L-xF) + \frac{1}{n}\Vol'(L-xF)\cdot (K_X+A_X(F)F))dx.$$ The fundamental theorem of calculus implies $$\int_0^{\tau(F)}\Vol'(L-xF)\cdot Fdx = L^n,$$ hence $$2(n-1)!\DF(\mX,\mL)  = A_X(F) \Vol(L) + \int_0^{\tau(F)} (n\mu\Vol(L-xF) + \Vol'(L-xF)\cdot K_X)dx,$$ which is our formula for $\beta(F)$. \end{proof}

\begin{corollary}\label{cor:onedirection} Valuative semistability implies integral K-semistability. \end{corollary}

It also follows that $\beta(F)$ generalises the usual $\beta$-invariant used in the study of Fano varieties.

\begin{corollary}\label{cor:fuj} We have $$\beta(F) =  A_X(F)\Vol(L) -\mu\int_0^{\tau(F)}\Vol(L-xF)dx + \int_0^{\tau(F)}\Vol'(L-xF)\cdot (\mu L +K_X)dx. $$ Thus when $L=-K_X$, $\beta(F)$ agrees with Fujita's invariant. \end{corollary}

\begin{proof}
Integrating by parts gives $$\int_0^{\tau(F)} (L-xF)^{n-1}.L dx =\left(1+\frac{1}{n}\right)\int_0^{\tau(F)}\Vol(L-xF)dx,$$ which provides the second interpretation of $\DF(\X,\L)$. When $L=-K_X$, the term involving $\mu L +K_X=0$ in the formula $\beta(F)$ vanishes, recovering Fujita's formula in this case since $\mu = \mu(X,-K_X)=1$. 
\end{proof}

We next turn to the norms involved. It seems most convenient to use $\JNA^{\NA}$ and $j(F)$, though one could use any of the Lipschitz equivalent norms.

\begin{lemma}\label{lemma:norm} We have $$\Vol(L)\JNA^{\NA}(\X,\L) = j(F).$$\end{lemma}

\begin{proof}

\cite[Lemma 7.7]{BHJ} gives $$\frac{\L.L^n}{L^n} = \lambda_{\max} = \tau(F),$$ which is one of the terms in interest in $$\JNA^{\NA}(\X,\L) = \frac{\L.L^n}{L^n} - \frac{\L^{n+1}}{(n+1)L^n}.$$ The remaining term can be understood through  Theorem \ref{thm:intersections}  from the leading order term of the weight polynomial $b_0$, giving $$\JNA^{\NA}(\X,\L) = \tau(F) - \frac{n!b_0}{L^n}.$$ Since the leading order term of the weight polynomial satisfies by Equation \eqref{eqn:weightpoly} $$b_0 = f_{n+1} + \lambda_{\min}a_0 = f_{n+1},$$ the equality $\lambda_{\min} = \tau(L)$ together with the equation for $f_{n+1}$ given by Propositon \ref{vol-1} provides $$\Vol(L)\JNA^{\NA}(\X,\L) =\int_0^{\tau(F)}(\Vol(L)-\Vol(L-xF))dx= j(F),$$ as required. \end{proof}

\subsection{The converse} We now show that valuative semistability implies integral K-semistability. Thus let $(\X,\L)$ be an integral test configuration. We fix a resolution of indeterminacy as follows. \begin{equation}\label{resolution}
\begin{tikzcd}
\Y \arrow[swap]{d}{q} \arrow{dr}{p} &  \\
X\times\pr^1 \arrow[dotted]{r}{} & \X
\end{tikzcd}
\end{equation} Denote by $\hat \X_0$ the strict transform  of $\X_0$ in $\Y$. Then by pulling back functions from $X$ to $\Y$, $\hat \X_0$ induces a divisorial valuation on $X$ which we denote $v_{\X_0}$ \cite[Section 4.2]{BHJ}. The induced valuation is independent of choice of resolution of indeterminacy. The filtration associated to $(\X,\L)$ can then be understood through the divisor $$p^*\L - q^*L = D.$$

\begin{lemma}\cite[Lemma 5.17]{BHJ} The filtration associated to $(\X,\L)$ can be described as $$\F^jV_k = \{ f \in V_k | v_{\X_0}(f) \geq -k\ord_{\hat\X_0}(D) +j\}.$$ Thus when $j \geq k\ord_{\hat\X_0}(D)$ we have $$\F^jV_k = H^0(X, kL-(k\ord_{\hat\X_0}(D) + j)v_{\X_0}),$$ with $\F^jV_k = V_k$ otherwise.
\end{lemma}

In fact Boucksom-Hisamoto-Jonsson's result is more general, but the resulting filtration simplifies when the central fibre $\X_0$ is integral. Note that $\X_0$ is dreamy, since the filtration $\F$ is finitely generated as it arises from a test configuration \cite[Proposition 2.15]{BHJ}. The maximal and minimal weights are thus the following.

\begin{corollary} We have \begin{align*}\lambda_{\max}&= \tau(v_{\X_0}) + \ord_{\hat \X_0}(D), \\ \lambda_{\min}&=\ord_{\hat \X_0}(D).\end{align*}\end{corollary}

Using this, we can relate the Donaldson-Futaki invariant and the norm.

\begin{proposition}\label{prop:converse} The Donaldson-Futaki and $\beta$-invariants agree u to a constant: $$2(n-1)!\DF(\X,\L) = \beta(v_{\X_0}).$$
\end{proposition} 

\begin{proof}

From the definition of the filtration and a change of variables in the integral, we have $$w(k) = \int_{0}^{\tau(v_{\X_0})} h^0(X, kL-xv_{\X_0})dx + k\ord_{\X_0}(D)h(k).$$ Adding a constant multiple of $kh(k)$ to the weight polynomial, which geometrically corresponds to adding a constant to the weight polynomial, leaves the Donaldson-Futaki invariant unaffected, so we may disregard the term $k\ord_{\X_0}(D)h(k)$. Since $v_{\X_0}$ is a dreamy prime divisor, we may apply the arguments of Section \ref{sec:MMP} to understand the integral $$f(k) = \int_{0}^{\tau(v_{\X_0})} h^0(X, kL-xv_{\X_0})dx.$$ Indeed, this is precisely the polynomial considered in Section \ref{sec:MMP}, showing by Proposition \ref{prop:relation-between-DF} that $2(n-1)!\DF(\X,\L) = \beta(v_{\X_0}).$ \end{proof}

A similar calculation applies to the norm, the details are the same as Lemma \ref{lemma:norm} and are thus left to the interested reader.

\begin{lemma}\label{lem:converse} We have $$\Vol(L)\JNA^{\NA}(\X,\L) = j(v_{\X_0}).$$ \end{lemma}

\begin{remark} In the case that the central fibre $\X_0$ is actually smooth or has orbifold singularities, the Donaldson-Futaki invariant of $(\X,\L)$ agrees with the classical Futaki invariant  of the induced holomorphic vector field on $\X_0$ by a result of Donaldson \cite[Proposition 2.2.2]{donaldson-toric}. Thus in this situation, the beta invariant of the induced divisorial valuation $\beta(v_{\X_0})$ also agrees with the classical Futaki invariant, since by our results $2(n-1)!\DF(\X,\L) = \beta(v_{\X_0}).$\end{remark}

\subsection{Equivariant K-polystability}

We now turn to the equivariant setting. The key notion is the slightly non-geometric notion of a \emph{product type} dreamy prime divisor, based on the definition due to Fujita in the Fano setting \cite[Definition 3.9]{fujita-hyperplane}. From Lemma \ref{associated-test-config} we obtain a test configuration, which we denote  $(\X_F,\L_F)$.

\begin{definition} We say that a dreamy prime divisor $F$ over $(X,L)$ is of \emph{product type} if its associated test configuration $(\X_F,\L_F)$ is a product test configuration. \end{definition}

The reason we have postponed this definition until the present section is that the definition relies on the correspondence between integral test configurations and dreamy prime divisors. What we have proven thusfar immediately produces:

\begin{corollary} A polarised variety $(X,L)$ is integrally K-polystable if and only $\beta(F)\geq 0$ for all dreamy prime divisors $F$ over $(X,L)$, with equality if and only if $F$ is of product type. \end{corollary}

We finally turn to $G$-equivariant K-polystability, with $G \subset \Aut_0(X,L)$ a connected algebraic group. We say that a dreamy prime divisor $F \subset Y$ is $G$\emph{-invariant} if there is a $G$ action on $Y$, making the map $Y \to X$ a $G$-invariant map, such that $F$ is itself a $G$-invariant divisor on $Y$ (by which we mean $G$ sends $F$ to itself rather than $F$ being contained in the fixed point locus of $G$). The following is a variant of work of Golota and Zhu \cite{golota, zhu}.

\begin{theorem} A polarised variety $(X,L)$ is $G$-equivariantly integrally K-polystable if and only if $\beta(F)\geq 0$ for all $G$-invariant dreamy prime divisors $F$ over $(X,L)$, with equality if and only if $F$ is of product type. \end{theorem}

\begin{proof} The claim follows from two facts. The first is that the valuation $v_{\X_0}$ associated to a $G$-equivariant integral test configuration $(\X,\L)$ is a $G$-invariant valuation. This follows directly from its construction. Indeed, taking a $G$-equivariant resolution of indeterminacy \begin{equation*} \begin{tikzcd}
\Y \arrow[swap]{d}{} \arrow{dr}{} &  \\
X\times\pr^1 \arrow[dotted]{r}{} & \X, \end{tikzcd}\end{equation*} one realises the proper transform $\hat \X_0 \subset \Y_0$ as a $G$-invariant divisor of $\Y$, implying that $v_{\X_0}$ is a $G$-invariant divisorial valuation on $X$, exactly as in Golota's proof in the Fano case \cite[Proposition 3.13]{golota}. The second is that the  integral test configuration associated with a $G$-invariant dreamy prime divisor is a $G$-equivariant test configuration, which, as noted by Zhu in the Fano setting \cite[Theorem 3.5]{zhu}, follows immediately from its definition as $$\X = \Proj_{\C} \bigoplus_{k \in \Z_{\geq 0}} \left( \bigoplus_{j \in \Z} t^{-j} \F^j V_k\right),$$ with the $G$-action induced from the natural action on $$\bigoplus_{k \in \Z_{\geq 0}} \bigoplus_{j \in \Z} t^{-j} \F^j V_k.$$\end{proof}

\section{Examples and properties}\label{sec:examples}

\subsection{Calabi-Yau and canonically polarised varieties}

It is a well-known result of Odaka that Calabi-Yau varieties and canonically polarised varieties are K-stable \cite{odaka-calabi}, and one can even show that they are uniformly K-stable \cite{twisted, BHJ}. It follows from Theorem \ref{thm:intromainthm} that they are, therefore, also valuatively stable. Nevertheless, it seems worth providing a direct proof as a demonstration of how to understand valuative stability. We will use the $\delta$-invariant $\delta(L)$, which is defined as \cite{fujita-odaka,blum-jonsson} $$\delta(L) = \inf_F \frac{A_X(F)\Vol(L)}{\int_0^{\infty} \Vol(L-xF) dx}.$$

\begin{theorem}\label{thm:CY} Let $(X,L)$ be a polarised variety and suppose that either

\begin{enumerate}[(i)]
\item $X$ has log terminal singularities $K_X=0$; or
\item $X$ has log canonical singularities and $L=K_X$.
\end{enumerate} Then $(X,L)$ is uniformly valuatively stable.
\end{theorem}

\begin{proof}
$(i)$ In this case, the invariant of interest simplifies to $$\beta_L(F) = A_X(F) L^n,$$ which is clearly non-negative. To show strict positivity, we use a result of Blum-Jonsson which implies that, since $X$ has log terminal singularities, $\delta(L)>0$ \cite[Theorem A]{blum-jonsson} and hence $$\beta(F) \geq \delta(L) S(F),$$ proving the result by Proposition \ref{prop:lipshitz-norms}.

$(ii)$ Using $\mu(X,K_X)=-1$, one calculates $$\beta_{K_X}(F) = A_X(F) (K_X)^n + \int_0^{\tau(F)} \Vol(K_X-xF)dx.$$ Since $X$ has log canonical singularities, $ A_X(F)  \geq 0$, and the result follows again by Lipschitz equivalence of $j(F)$ and $\int_0^{\tau(F)} \Vol(K_X-xF)dx$ proved in Proposition \ref{prop:lipshitz-norms}.
\end{proof}

We next turn to a sufficient criterion involving $\delta = \delta(L)$. 

\begin{theorem}\label{thm:bodydelta} Write $\delta(L) - \mu(L)=(n+1)\gamma(L)$, and suppose $$(\mu(L) + \gamma(L)) L + K_X \textrm{ is effective}.$$ Then $(X,L)$ is uniformly valuatively stable.
\end{theorem}

\begin{proof}
We use the formulation of Corollary \ref{cor:fuj}, which demonstrates that $$\beta(F) =  A_X(F)\Vol(L) -\mu\int_0^{\infty}\Vol(L-xF)dx + \int_0^{\infty}\Vol'(L-xF)\cdot (\mu L +K_X)dx.$$ We write $(n+1)\gamma = (n+1)\gamma' + \epsilon$, with $\epsilon>0$ chosen so that $(\mu + \gamma') L + K_X \textrm{ is effective};$ such a choice exists since effectivity is an open condition in the N\'eron-Severi group.

Since by definition of $\delta $ there is a lower bound $$A_X(F)\Vol(L) \geq \delta\int_0^{\infty}\Vol(L-xF)dx,$$ the $\beta$-invariant has a lower bound of the form \begin{align*}\beta(F) &\geq (\delta-\mu)\int_0^{\infty}\Vol(L-xF)dx + \int_0^{\infty}\Vol'(L-xF)\cdot (\mu L +K_X)dx, \\ &=  (n+1)\gamma\int_0^{\infty}\Vol(L-xF)dx + \int_0^{\infty}\Vol'(L-xF)\cdot (\mu L +K_X)dx.\end{align*} Note from integration by parts as in the proof of Corollary \ref{cor:fuj} that $$\int_0^{\infty}\Vol(L-xF)dx = (n+1) \int_0^{\infty}\Vol'(L-xF)\cdot Ldx.$$ Thus $$\beta(F) \geq \epsilon\int_0^{\infty}\Vol(L-xF)dx  +   \int_0^{\infty}\Vol'(L-xF)\cdot ((\mu + \gamma') L +K_X)dx.$$ The proof is concluded by noting that since $(\mu + \gamma) L + K_X$ is assumed effective, the derivative of the volume in this direction is non-negative \cite[Corollary C]{BFJ}, with the term $\epsilon\int_0^{\infty}\Vol(L-xF)dx$ providing uniform valuative stability by Proposition \ref{prop:lipshitz-norms}.
\end{proof}

The same proof shows that provided $(\mu + \gamma) L + K_X$ is nef, $(X,L)$ is valuatively semistable. We remark that in the case $L=-K_X$, this result recovers Fujita-Odaka's result that if $\delta(-K_X) > 1$, then $(X,-K_X)$ is uniformly valutively stable (hence uniformly K-stable by the  work of Fujita and Li  \cite{fujita-valuative, li}). Thus Theorem \ref{thm:bodydelta} can be through of as a generalisation of Fujita-Odaka's work to more general polarised varieties.

\begin{remark} It is interesting to note that the dreaminess hypothesis is irrelevant in all of our sufficient criteria for valuative stability. \end{remark}

\subsection{Valuatively unstable varieties}

Since valuative stability implies the classical Futaki invariant vanishes, one obtains many examples of valuatively unstable varieties. It seems worth providing one calculation of this fact directly. The example we choose is the blow up of $\pr^2$ at a point, which is K-unstable with respect to any polarisation. We show that it is even valuatively unstable.

\begin{proposition} $\Bl_{p}\pr^2$ is valuatively unstable with respect to any polarisation. \end{proposition}

\begin{proof} We show that the exceptional divisor destabilises. Note that as $X$ is Fano, the exceptional divisor is dreamy with respect to any polarisation by Example \ref{dreamy-fano}.

Let $H$ be the pullback of the hyperplane class on $\pr^2$ to $\Bl_p\pr^2$, and let $E$ be the exceptional divisor. The ample divisors are of the form $$xH-yE, \qquad x > y \geq 0;$$ this line bundle is nef when $x=y$. The big divisors are of the form $$xH-yE, \qquad  x \geq 0.$$ 

By \cite[Example 2.2.46]{Lazarsfeld1}, the volume is given by \begin{align*} \Vol(xH-yE) &= x^2-y^2, \qquad \text{ when }  x > y \geq 0, \\ &= x^2, \text{ when } \qquad x \geq 0, y \leq  0.\end{align*}  Thus the pseudoeffective threshold is  $$\tau_{xH-yE}(E) = x-y.$$ 

A calculation then gives that $$\beta_{xH-yE}(E) = x^2-y^2+\frac{6x-2y}{x^2-y^2}\int_0^{x-y}(x^2-y^2-2yz-z^2)dz + \int_0^{x-y}(-6x+2y+2z)dz.$$ By homogeneity, in considering the sign of this invariant, we may assume that $x=3$. Then we have $$\beta_{3H-yE}(E)= -\frac{4 (y - 3)^2 y}{3 (y + 3)},$$ which is negative for $0<y<3$. Note that when $y=1$, the polarisation is given by the anticanonical class, $\beta_{3H-E}(E) = -4/3$ and one can check this agrees with the calculation of Fujita's $\beta$-invariant.

\end{proof}

\begin{example}

It is in general not difficult to produce valuatively unstable varieties with discrete automorphism group. For example, let $E$ be a simple unstable vector bundle over a polarised Riemann surface $(B,L)$ of genus at least one. Then we claim $(\pr(E),kL+\scO_{\pr(E)}(1))$ is valuatively unstable for all $k \gg 0$. Indeed, any subbundle $F \subset E$ induces a test configuration $$(\pr(\E), \scO_{\pr(\E)}(1))\to B \times \C$$ with  $\E$ a bundle over $B \times \C$ which satisfies $$\E_0 \cong F \oplus E/F.$$ Thus $\pr(\E)_0$ is smooth, hence integral. Since $B$ has dimension one and $E$ is by hypothesis unstable, there exists a destabilising subbundle $F \subset E$. It then follows from a result Ross-Thomas that the Donaldson-Futaki invariant of $(\pr(E),kL+\scO_{\pr(E)}(1))$ is strictly negative for $k\gg 0$  \cite[Section 5.4]{ross-thomas}, and hence by Theorem \ref{thm:intromainthm} is also valuatively unstable. The associated divisorial valuation is induced by $\pr(\E)_0$ through the constructions. 

\end{example}

\subsection{Bounds on the alpha invariant}

Recall that the alpha invariant of $(X,L)$ is defined as $$\alpha(X,L) = \inf_{D \in |kL|} \lct\left(X,\frac{1}{k}D\right),$$ where $$ \lct(X,D) = \sup\{ t\in \R_{>0}\ |  \ (X,tD) \text{ is log canonical}\}.$$ Fujita-Odaka have shown that K-semistable Fano varieties have alpha invariant bounded below by $\frac{1}{n+1}$ \cite{fujita-odaka}. The analogue for for general K-semistable varieties is the following. Let us say that $(X,L)$ is \emph{strongly valuatively semistable} if $\beta(F)\geq 0$ for all prime divisors over $X$, not necessarily dreamy.

\begin{theorem}\label{thm:alphabounds}
Suppose $(X,L)$ is a strongly valuatively semistable. Then $$\alpha(X,L) \geq \frac{\mu(X,L)}{n+1}.$$

\end{theorem}

\begin{proof}

Consider a divisor $D \in |kL|$. Note that $\lct(X,D) = A_X(D)$ \cite[Proposition 8.5]{singularitiesofpairs}. By working with $D$ as a $\Q$-divisor, we may assume $k=1$, and hence $\tau_L(D) = 1.$ It follows that \begin{align*}\beta_L(D) =& \lct(X,D) L^n + n \mu \int_{0}^{1} (1-x)^n\Vol(L)dx + \int_0^{1}(1-x)^{n-1}\Vol'(L)\cdot K_X dx, \\ &= \lct(X,D) L^n +\frac{L^n}{n+1} + L^{n-1}.K_X. \end{align*} By hypothesis, $(X,L)$ is a strongly valuatively semistable, and hence $\beta_L(D) \geq 0$. Thus $$ \lct(X,D) \geq \frac{\mu(X,L)}{n+1}.$$ The result follows by taking the infimum over all such $D$. \end{proof}

In the toric setting, we can replace \emph{strong} valuative semistability with valuative semistability.

\begin{corollary} Suppose $(X,L)$ is a valuatively semistable toric variety. Then $$\alpha(X,L) \geq \frac{\mu(X,L)}{n+1}.$$\end{corollary}

\begin{proof}

By a result of Cheltsov-Shramov, the alpha invariant $\alpha(X,L)$ can be computed using toric divisors $D$ on $X$ \cite[Lemma 5.1]{cheltsov-shramov}; the proof degenerates an arbitrary divisor in the linear system $|kL|$ to a linearly equivalent one which is invariant under each $\C^*$-action inside the torus successively, and uses that the log canonical threshold can only drop when taking such a limit. 

Thus let $D$ be a toric divisor. Since $D$ is toric, one can take a toric resolution of singularities $Y \to X$ of $(X,D)$. Since $Y$ is toric, the proper transform of $D$ on $Y$ is thus dreamy. This means that in the previous argument, we can weaken strong valuative semistability to simply valuative semistability, proving the result. \end{proof}

These results are of course only interesting when $\mu(X,L) > 0$; this is automatic for toric varieties. Since K-semistability implies valuative semistability, the above applies to K-semistable toric varieties.

\section{Toric varieties}\label{sec:toric}

Let $X=X_\Sigma$ be a compact $n$--dimensional toric variety associated to a fan $\Sigma\subset N_\R$ where $N$ is the lattice of circle subgroups of the torus $T_N\simeq (\C^*)^n$. We first consider valuative stability for toric divisors on $X$ itself.  We follow the notation of \cite{CLS:toricbook} and thus denote $M=N^*$ the lattice of characters of $T_N$ and $\Sigma(1)$ the set of rays of $\Sigma$. Each $\rho\in \Sigma(1)$ determines both a prime divisor $D_\rho$ and  an element  $u_\rho\in N$, namely the (unique) primitive vector in $\rho \cap N$, see \cite[\S 4.1]{CLS:toricbook}. The ample line bundle corresponds to a full dimensional lattice polytope $P=P_L$ (uniquely determined by the linear equivalence of $L$ up to a translation) whose fan is $\Sigma$.  

We will consider the case where $\Sigma$ is simplicial, to ensure that $X_\Sigma$ is a Mori dream space \cite{HK}. In that case, $X_\Sigma$ has at worst orbifold singularities; equivalently it is normal and $\Q$--factorial (i.e any Weil divisor is $\Q$-Cartier) \cite[p.549]{CLS:toricbook}.

The torus relative Futaki invariant of a toric polarised variety $(X, L)$ can be identified with the difference of the barycentres of the polytope $P_L$ and its boundary as highlighted in \cite{donaldson-toric}. More precisely, $N_\R$ is the Lie algebra of the real compact torus lying in $T_N$ and thus parameterises the space of toric holomorphic vector fields on $X$. Alternatively, elements of $N_\R$ are identified with affine-linear functions on $P_L$. The Futaki invariant in this setting coincides, up to a positive factor, with a function $\mbox{Fut} : N_\R \ra \R$ defined at $f\in N_\R$ by
\begin{equation}\label{e:FutToric}
 \mbox{Fut}(f) = \frac{1}{\Vol_{M} (\partial P_{L})}\int_{\partial P} f \varpi^{\partial P} - \frac{1}{ \Vol_M(P_L)} \int_P f \varpi^M
\end{equation}where $\Vol_{M} (\partial P_{L})= \int_{\partial P} \varpi^{\partial P}$, $\Vol_M(P_L)=\int_P  \varpi^M$ and the measures $\varpi^M$ and  $\varpi^{\partial P}$ will be defined in the next paragraph. It follows from this formula~\eqref{e:FutToric}, since $f$ is affine linear, that the vanishing of the Futaki invariant for all such functions $f$ is equivalent to the barycentres of $P$ and $\partial P$ being the same, as claimed.

We now recall the definition of the measures used to compute these barycentres.  With the lattice $M$ comes a unique measure on $M_\R$, the Lebesgue measure, scaled so that $M_\R/ M$ has volume $1$. The same happens for each \emph{rational} subspaces $H$ of $M_\R$, i.e., those such that $M\cap H$ span $H$. In particular, this is true for the hyperplane $\rho^\lor + m$ where $\rho\in \Sigma(1)$ and $m\in M$ and more generally by translation for $m\in M_\R$. These measures will be encoded by volume forms, say $\varpi^M \in \Lambda^n(M_\R)$ and $\varpi_\rho\in \Lambda^{n-1}(\rho^\lor + m)$. We define $\varpi^{\partial P}$ so that its restriction to the facet of $P$ corresponding to $\rho\in \Sigma(1)$ equals $\varpi_\rho$. 

\begin{remark} Note that along $\rho^\lor + m$, we have 
 \begin{equation}\label{e:BndMEASURE} u_\rho\wedge \varpi_\rho =- \varpi^M\end{equation} so the Futaki invariant formula \eqref{e:FutToric} coincides with Donaldson's formula in \cite{donaldson-toric}. 
\end{remark}

Recall that a Cartier divisor $D$ on $X$ is necessary of the form $D\sim \sum_{\rho\in \Sigma(1)} a_\rho D_\rho$ and is associated to a polyhedron $$P_D := \{ m\in M_\R \,|\, \langle m,u_\rho \rangle \geq -a_\rho,\;\; \rho \in \Sigma(1)\}.$$ For $\rho\in \Sigma(1)$, we define the (possibly empty) convex set $F^{P_D}_{\rho}:= P_D\cap  \{ m\in M_\R \,|\, \langle m,u_\rho \rangle = -a_\rho \}$. The boundary of $P_D$, if non-empty, is an union of set of the forms $F^{P_D}_{\rho}$. If $D$ is ample $F^{P_D}_{\rho}$ is a non empty codimension $1$ face for any $\rho\in \Sigma(1)$, otherwise $F^{P_D}_{\rho}$ might be of lower dimension or empty. If $\dim F^{P_D}_{\rho} < \dim M_\R -1$ then $F^{P_D}_{\rho}$ has no contribution on the following numerical invariant: $$\Vol_{M} (\partial P_D) := \sum_{ \rho\in \Sigma(1)} \int_{F^{P_D}_{\rho}} \varpi_\rho.$$ Here, if $\dim F^{P_D}_{\rho} = n-1$, then $F^{P_D}_{\rho}$ is endowed with the orientation coming from the inclusion $F^{P_D}_{\rho} \subset P_D$; that is, $\int_{F^{P_D}_{\rho}} \varpi_\rho>0$.

Thanks to \cite[Theorem 4.1.3]{CLS:toricbook}, $\Vol_{M} (\partial P_D)$ depends on $D$ only up to linear equivalence. The notation $\Vol_{M} (\partial P_D)$ might be misleading because when $\dim P_D = n-1$, the usual boundary $\partial P_D = \emptyset$ but $\Vol_{M} (\partial P_D)\neq 0$. This situation does not happen when $D$ is big \cite[p.427]{CLS:toricbook}.

\begin{lemma}\label{l:FormulaVolume} Let $L$ be an ample line bundle over $X= X_\Sigma$. For any $x\in [0,+\infty)$ and $\rho \in\Sigma(1)$, we have $\Vol (L) = n! \Vol_{M} (P)$, 
$$\Vol (L-xD_\rho)  = n! \Vol_{M} (P_L \cap (\langle , u_\rho\rangle \geq x-a_\rho )),$$
 and for $x\in [0,\tau_L(D_\rho))$ $$\Vol' (L-xD_\rho)\cdot (K_X) = - n!\Vol_{M} (\partial (P_{L-xD_\rho})).$$
\end{lemma}

\begin{proof}
 The first two statements are proved by substituing $-K_X$ with  $L$ in Fujita's argument of Claim 6.2 in \cite{fujita-plms}. For the last statement, note that for $x\in [0,\tau_L(D_\rho))$, the divisor $L-xD_\rho$ is big and thus $P_x:= P_{L-xD_\rho}$ is a full dimensional polytope, and the same hold for $P_{x,t}:= P_{L-xD_\rho + tK_X}$ when $t\in \R$ is sufficiently small. We pick $a_\lambda(x) \in \R$ so that $$L-xD_\rho \sim \sum_{\lambda \in \Sigma(1)}  a_\lambda(x) D_\lambda$$ and recall that $-K_X\sim \sum_{\lambda \in \Sigma(1)} D_\lambda$ \cite[Theorem 8.2.3]{CLS:toricbook}. Hence $$L-xD_\rho + tK_X \sim \sum_{\lambda \in \Sigma(1)}  (a_\lambda(x) -t) D_\lambda.$$   
 
 Assuming $t>0$,  $P_{x,t} \subseteq P_{x}$ and if $\epsilon>0$ is small enough the combinatorial type of $P_{x,t}$ does not depend on $t\in [0,\epsilon)$. Then
 \begin{equation}\label{eqn:measure}
  \begin{split}
 \Vol_{M} (P_{x,t}) & =\Vol_{M} (P_{x}) - \Vol_{M} (P_{x}\backslash P_{x,t})\\
  & =\Vol_{M} (P_{x}) - \int_0^t \Vol_{M}(\partial P_{x,s}) ds. 
  \end{split}
 \end{equation}
%
 
 The derivative with respect to $t$ of the right hand side is $-\Vol_{M} (\partial (P_{L-xD_\rho}))$. To conclude the proof, one argues that the left hand side is $C^1$ so that we may compute the derivative at $t=0$ using any converging sequence (so assuming $t>0$ is sufficient). \end{proof}

\begin{corollary} For a compact toric variety $X_\Sigma$ and ample bundle $L$, we have that $$\mu(X,L) = \frac{\Vol_{M} (\partial P_{L})}{n \Vol_M(P_L)}$$ and for any prime toric divisor $D= D_\rho$, we have  $$\beta_L(D) = \Vol_M(P_L) + \frac{\Vol_{M} (\partial P_{L})}{\Vol_M(P_L)} \int_{0}^{\tau}  \Vol_M(P_{L-xD}) dx -  \int_{0}^{\tau}\Vol_{M} (\partial (P_{L-xD})dx.$$ \end{corollary}

\begin{lemma}\label{l:barycenters=divisorial} Fix $\lambda \in \Sigma(1)$ and an ample line bundle $L$. We have $$\frac{\beta_{L}(D_\lambda)}{\Vol_{M}(\partial P_L)} = \langle b_{P_L}-b_{\partial P_L}, u_\lambda \rangle$$ where $b_{P_L}:=\barycenter(P_L,\varpi^M)$ and $b_{\partial P_L}:=\barycenter(\partial P_L,\varpi_\Sigma)$ are the barycentres.
\end{lemma}

\begin{remark}
 Whenever $L= -K_X$, using $-K_X\sim \sum_{\lambda \in \Sigma(1)} D_\lambda$ one can check that 
 $$ \left( 1+\frac{1}{n}\right) b_P = b_{\partial P}$$
 and we recover Fujita's formula \cite[Theorem 6.1]{fujita-plms}.
\end{remark}

    \begin{proof} We denote $P=P_L$, $P_x := P_{L-xD_\lambda}$ and $\tau= \tau_{D_\lambda}(L)$. Note that $$P=\bigsqcup_{0\leq x\leq \tau} F_{\lambda,x} $$ where $F_{\lambda,x} = P\cap \{ m\in M_\R\,|\, \langle m, u_\lambda \rangle = -a_\lambda +x \} $ and $L= \sum_{\rho \in \Sigma(1)} a_\rho D_\rho$. Then, $$\int_P \langle m, u_\lambda \rangle \varpi^M = \int_0^\tau (x-a_\lambda) \left(\int_{F_{\lambda,x}} \varpi_\lambda \right)dx = \int_0^\tau (x-a_\lambda)\Vol_\lambda (F_{\lambda,x}) dx$$ where $\Vol_\lambda (F_{\lambda,x}) :=  \int_{F_{\lambda,x}} \varpi_\lambda$. Now, we have $$\Vol_\lambda (F_{\lambda,x}) = \frac{-1}{n!} \frac{d}{dx}\Vol(L-xD_\lambda) = - \frac{d}{dx} \Vol_M(P_{L-xD_\lambda}).$$ Hence, using integration by parts, we obtain 
    \begin{equation}\label{e:barycenterP}
     \int_P \langle m, u_\lambda \rangle \varpi^M = -a_\lambda \Vol_M(P) + \int_0^\tau \Vol_M(P_{L-xD_\lambda})\, dx.
    \end{equation} 
    
    We define the affine-linear function $\ell_\lambda (m) = \langle m, u_\lambda \rangle +a_\lambda$ so that $0 = \min_P \{\ell_\lambda\}$ and thus $\tau = \max_P \{\ell_\lambda\}$. The boundary barycentre gives \begin{equation}\label{e:barycenterPartialPAff} \langle b_{\partial P}, u_\lambda \rangle + a_\lambda = \frac{1}{\int_{\partial P} \varpi_\Sigma}\int_{\partial P} \ell_\lambda \, \varpi_\Sigma.\end{equation} As before we can write $$\int_{\partial P} \ell_\lambda \,\varpi_\Sigma = \int_{0}^\tau x \Vol_{\lambda,\varpi}^{n-2}(\partial P \cap \{\ell_\lambda= x \}) dx$$ where $\Vol_{\lambda,\varpi}^{n-2}(\partial P \cap \{\ell_\lambda= x \})$ is the volume of $\partial P \cap \{\ell_\lambda= x \}$ with respect to a volume form and orientation that are cumbersome to define but satisfy $$-\frac{d}{dx} \Vol_{M} (\partial P \cap \{\ell_\lambda > x\}) = \Vol_{\lambda,\varpi}^{n-2}(\partial P \cap \{\ell_\lambda= x \}). $$ Using that $\Vol_{M} (\partial P \cap \{\ell_\lambda > x\}) = \Vol_{M} (\partial P_x) - \Vol_\lambda (F_{\lambda,x})$ and integration by parts produces
    
     \begin{equation}\label{e:barycenterPartialP}
  \begin{split}
   \int_{\partial P} \ell_\lambda \,\varpi_\Sigma &= \int_0^\tau d \left( x \left[ \Vol_{M}(\partial P_x) -\Vol_\lambda (F_{\lambda,x})\right]\right) \; + \;  \int_0^\tau \left( \Vol_{M}(\partial P_x) -\Vol_\lambda (F_{\lambda,x})\right) dx \\
   &= -\tau \Vol_{M}(\partial P_\tau) + \tau \Vol_\lambda (F_{\lambda,\tau}) \,+ \,  \int_0^\tau \Vol_{M}(\partial P_x) dx  \,- \,  \Vol P\\
   &=   \int_0^\tau  \Vol_{M}(\partial P_x) dx  \,- \,  \Vol P.
  \end{split}\end{equation} Here, $\Vol_{M}(\partial P_\tau) :=\lim_{x\ra \tau} \Vol_{M}(\partial P_x)$ and for the last line we have used that $\Vol_{M}(\partial P_\tau) = \Vol_\lambda (F_{\lambda,\tau})$. Indeed, if there is no parallel facets to $F^P_\lambda$ in $P$ then both vanish, while if $F_{\lambda,\tau}$ is a facet of $P$ then $\Vol_{M}(\partial P_\tau) = \Vol_\lambda (F_{\lambda,\tau})$. We get the result we seek by combining \eqref{e:barycenterPartialP}, \eqref{e:barycenterPartialPAff} and \eqref{e:barycenterP}. \end{proof}

\subsection{Star subdivision}

We next consider the case of a toric prime divisor on a birational toric model $Y \to X$. Assume again that $X=X_{\Sigma}$ is a compact toric manifold with (complete and simplicial) fan $\Sigma$. Consider $u_\nu \in N$ a primitive lattice element with ray $\nu = \mbox{Cone}(u_\nu)$ and the associated star subdivision $\Sigma^*_\nu:=\Sigma^*(\nu)$ as defined in \cite[p.515]{CLS:toricbook} which is refinement of $\Sigma$. Note that $$\Sigma^*_\nu(1) = \Sigma(1) \cup \{\nu\}$$ as proved in \cite[Theorem 11.1.6]{CLS:toricbook} and that $\Sigma^*_\nu$ is complete and simplicial. Thus $\Sigma^*_\nu$ is associated to a (normal) compact toric variety say $Y$, with at worst orbifold singularities \cite[Theorem 11.4.8]{CLS:toricbook}, and on which $\nu$ defines $D_\nu$, a prime Weil divisor. This divisor is the exceptional divisor of the toric morphism $\psi : Y\ra X$ induced from the inclusion $\underline{\psi} : \Sigma^*_\nu \ra \Sigma$.

We assume that $\nu \notin \Sigma(1)$ (actually the case $\nu \in \Sigma(1)$ coincides with what we have done above) and denote $\sigma$ the cone of $\Sigma$ of minimal dimension, say $r$, among those containing $\nu$. Thus $\sigma(1) = \{ u_1,\dots, u_r \}$ and there exists $c_i\in\N^*$ such that $u_\nu = \sum_{i=1}^r c_iu_i$. 
Using \cite[Lemma 11.4.10]{CLS:toricbook}, we have that $$K_Y +D_\nu = \psi^*K_X + A_X(D_\nu) D_\nu $$ where $A_X(D_\nu) = \sum_{i=1}^r c_i$.

Letting $L = \sum_{\rho \in \Sigma} a_\rho D_\rho$ be an ample line bundle over $X$, we denote $\varphi_L : \Sigma \ra \R$, the support function of $L$ (so that $a_\rho = -\varphi_L(u_\rho)$). Using the fact that the support function of $\psi^*L =:H$ is also $\varphi_L$ (see the proof of \cite[Lemma 11.4.10]{CLS:toricbook}) we obtain that $$H=  \left(-\sum_{i=1}^r c_i \varphi_L(u_i) \right)D_\nu +\sum_{\rho \in \Sigma} a_\rho D_\rho.$$ Moreover, since $L$ is ample on $X$ there are points $m \in P_L$ such that $\langle m, u_i \rangle=  \varphi_L(u_i),$ $\forall u_i\in \sigma(1)$. More precisely, the hyperplane $A:= \{ m\in M_\R \,|\, \langle m, u_\nu\rangle = \sum_{i=1}^r c_i \varphi_L(u_i) \}$ intersects $P_L$ in a face $F$ whose normal cone is $\sigma$. Denote the affine-linear function $$\ell_\nu(m) := \langle m, u_\nu\rangle -\sum_{i=1}^r c_i \varphi_L(u_i).$$ We have that $P_L \subset \{\ell_\nu \geq 0\}$ and thus $P_H=P_L$.

Observe that in the proof of Lemma~\ref{l:barycenters=divisorial} does not use that $\mbox{Cone}(u_\lambda)\in\Sigma$, hence recycling it in our case gives
\begin{equation}\label{e:betaTORICdivisors}
\frac{\beta_{L}(D_\nu)}{\Vol_{M}(\partial P_L)} = (A_X(D_\nu) -1)\frac{\Vol_M P_L}{\Vol_{M}(\partial P_L)} + \langle b_{P_L}-b_{\partial P_L}, u_\nu \rangle. 
\end{equation}

\begin{corollary}\label{c:toricFover} Assume that the (torus relative) Futaki invariant of $(X,L)$ vanishes identically (i.e $b_{P_L}=b_{\partial P_L}$). To any primitive vector $\nu \in N$ there is an associated a valuation $D_\nu$, defined via the refinement $\Sigma^*_\nu$ as above, which satisfies $\beta_{L}(D_\nu) \geq 0$. Moreover, equality  holds if $\nu\in \Sigma(1)$.  
\end{corollary}

Suppose now that $Y$ is a compact toric variety $Y$ and $\psi: Y\ra X$ is a proper birational toric morphism with exceptional (toric prime) divisor $F$.  Note that in that situation, $Y$ is associated to a fan, $Y= X_{\widetilde{\Sigma}}$, and the associated map $\underline{\psi}:\widetilde{\Sigma} \ra \Sigma$ is a refinement. Also, $\widetilde{\Sigma}$ must be of the form above using the description of the exceptional sets \cite[Proposition 11.1.10]{CLS:toricbook}. Thus Corollary \ref{c:toricFover} gives a complete description of toric valuative stability.

\begin{proof} [Proof of Theorem~\ref{prop:toricVal=Fut}.] Assume $\beta_{L}(F) \geq 0$ for any toric prime divisor $F$ over $X$. Because $X$ compact, $\Sigma$ is complete and thus some positive real numbers $t_\rho >0$ satisfy $$\sum_{\rho \in \Sigma(1)} t_\rho u_\rho =0.$$ By linearity of the Futaki invariant and by Lemma~\ref{l:barycenters=divisorial}, we have $$0\leq  \sum_{\rho \in \Sigma(1)} t_\rho \beta_{L}(D_\rho) = \Vol_{M}(\partial P_L)\, \sum_{\rho \in \Sigma(1)} t_\rho \mbox{Fut}_{L}(u_\rho) = 0.$$ Thus, $\beta_{L}(D_\rho)=0$ for any $\rho \in \Sigma(1)$. The converse is Corollary~\ref{c:toricFover}.\end{proof}

\section{Declarations}

\subsection*{Funding} RD was funded by a Royal Society University Research Fellowship. EL held a visiting position at Churchill College in Cambridge, she was also supported by CNRS (IEA-International Emerging Actions grant number 295351) and a CIMI mobility grant. 

\subsection*{Conflicts of interest/Competing interests } On behalf of all authors, the corresponding author states that there is no conflict of interest.

\subsection*{Availability of data and material} Not applicable.

\subsection*{Code availability } Not applicable.

\subsection*{Authors' contributions } Both authors contributed equally.

\bibliographystyle{abbrv}


\begin{thebibliography}{DDDD}

\bibitem{ADL} {\scshape C. Arezzo, A. Della Vedova and G. La Nave}, \emph{Singularities and K-semistability.} Int. Math. Res. Not. IMRN 2012, no. 4, 849-869.

\bibitem{AZ}
   {\scshape H. Ahmadinezhad and Z. Zhuang},
    \emph{K-stability of Fano varieties via admissible flags.} arXiv:2003.13788, 38pp.

\bibitem{berman} {\scshape R. Berman}, \emph{K-polystability of $\Q$-Fano varieties admitting K\"ahler-Einstein metrics.} Invent. Math. 203 (2016), no. 3, 973-1025.

\bibitem{BDL}  {\scshape R. Berman, T. Darvas and C. H. Lu}, \emph{Convexity of the extended K-energy and the large time behaviour of the weak Calabi flow.} To appear in Ann. Sci. Ec. Norm. Sup\'er.

\bibitem{BCHM}
   {\scshape C. Birkar, P. Cascini, C. Hacon and J. McKernan},
    \emph{Existence of minimal models for varieties of log general type.} J. Amer. Math. Soc. 23 (2010), no. 2, 405-468. 
    
    \bibitem{blum-liu-zhou}
   {\scshape H. Blum, Y. Liu and C. Zhou},
    \emph{    Optimal destabilization of K-unstable Fano varieties via stability thresholds.}  	To appear in Geom. Topol.
\bibitem{blum-jonsson}
   {\scshape H. Blum and M. Jonsson},
    \emph{Thresholds, valuations, and K-stability. } Adv. Math. 365 (2020), 107062, 57 pp. 

\bibitem{blum-xu}
   {\scshape H. Blum and C. Xu},
    \emph{Uniqueness of K-polystable degenerations of Fano varieties.} Ann. of Math. (2) 190 (2019), no. 2, 609-656. 

\bibitem{BFJ}
   {\scshape S. Boucksom, C. Favre and M. Jonsson},
    \emph{Differentiability of volumes of divisors and a problem of Teissier}.  J. Algebraic Geom. 18 (2009), no. 2, 279-308. 





   \bibitem{BHJ}
   {\scshape S. Boucksom, T. Hisamoto and M. Jonsson},
    \emph{Uniform K-stability, Duistermaat–Heckman measures and singularities of pairs}.  Ann. Inst. Fourier (Grenoble) 67 (2017), no. 2, 743-841. 
    
 
 
 \bibitem{BJ}
   {\scshape S. Boucksom and M. Jonsson},
    \emph{A non-Archimedean approach to K-stability, I: Metric geometry of spaces of test configurations and valuations}.   arXiv:1805.11160


       \bibitem{cheltsov-shramov}
   {\scshape I. Cheltsov and C. Shramov},
    \emph{Log-canonical thresholds for nonsingular Fano threefolds} (with an appendix by J.-P. Demailly). Uspekhi Mat. Nauk 63 (2008), no. 5(383), 73-180.

\bibitem{CDS} {\scshape X. Chen, S. Donaldson and S. Sun},
\emph{K\"ahler-Einstein metrics on Fano manifolds. I, II, III}
J. Amer. Math. Soc. 28 (2015), no. 1, 183-197, 199-234, 235-278.


   \bibitem{chen-cheng}
   {\scshape X. Chen and J. Cheng},
   \emph{On the constant scalar curvature K\"ahler metrics  (II)— existence results}. J. Amer. Math. Soc. 34 (2021), no. 4, 937-1009. 
   
          \bibitem{chen-sun}
   {\scshape X. Chen and S. Sun},
    \emph{Calabi flow, geodesic rays, and uniqueness of constant scalar curvature K\"ahler metrics.} Ann. of Math. (2) 180 (2014), no. 2, 407-454. 
    
    
       \bibitem{codogni-patakfalvi}
   {\scshape G. Codogni and Zs. Patakfalvi},
   \emph{Positivity of the CM line bundle for families of K-stable klt Fano varieties}.  To appear in Invent. Math.

   
\bibitem{CLS:toricbook}
	{\scshape D. Cox, B. Little and H. Schenck},
	{\it Toric varieties}, Graduate Studies in Mathematics 124, AMS, (2011). 

       \bibitem{TD}
   {\scshape T. Delcroix},
   \emph{Uniform K-stability of polarized spherical varieties}.  	arXiv:2009.06463, 25pp.

            \bibitem{alpha}
{\scshape R. Dervan},
 \emph{Alpha invariants and K-stability for general polarizations of Fano varieties.} Int. Math. Res. Not. IMRN 2015, no. 16, 7162-7189. 

        \bibitem{twisted}
{\scshape R. Dervan}, \emph{Uniform stability of twisted constant scalar curvature K\"ahler metrics. } Int. Math. Res. Not. IMRN 2016, no. 15, 4728-4783. 

    


        \bibitem{mabuchi}
{\scshape R. Dervan}, \emph{Alpha invariants and coercivity of the Mabuchi functional on Fano manifolds.  } Ann. Fac. Sci. Toulouse Math. (6) 25 (2016), no. 4, 919-934. 


   \bibitem{donaldson-toric}
   {\scshape S. K. Donaldson},
    \emph{Scalar curvature and stability of toric varieties.}. J. Differential Geom. 62 (2002), no. 2, 289-349. 

 
   \bibitem{fujita-plms}
   {\scshape K. Fujita},
   \emph{On $K$-stability and the volume functions of $\Q$-Fano varieties}. Proc. Lond. Math. Soc. (3) 113 (2016), no. 5, 541-582. 

   \bibitem{fujita-pems}
   {\scshape K. Fujita},
    \emph{Examples of K-unstable Fano manifolds with the Picard number one}.  Proc. Edinb. Math. Soc. (2) 60 (2017), no. 4, 881-891. 

   
   \bibitem{fujita-hyperplane}
   {\scshape K. Fujita},
   \emph{K-stability of log Fano hyperplane arrangements}. To appear in J. Alg. Geom.


   \bibitem{fujita-valuative}
   {\scshape K. Fujita},
    \emph{A valuative criterion for uniform K-stability of $\Q$--Fano varieties}.  J. Reine Angew. Math. 751 (2019), 309-338. 
    
       \bibitem{fujita-alpha}
   {\scshape K. Fujita},
    \emph{K-stability of Fano manifolds with not small alpha invariants.} J. Inst. Math. Jussieu 18 (2019), no. 3, 519-530. 
    
    
       \bibitem{fujita-maeda}
   {\scshape K. Fujita},
    \emph{Toward criteria for K-stability of log Fano pairs.} Proceedings of 64th Algebra Symposium (2019).


       \bibitem{fujita-proceedings}
   {\scshape K. Fujita},
    \emph{On K-polystability for log del Pezzo pairs of Maeda type}  To appear in Acta Math. Vietnam.
    
     
    
    \bibitem{fujita-odaka}
   {\scshape K. Fujita and Y. Odaka},
    \emph{On the K-stability of Fano varieties and anticanonical divisors.} Tohoku Math. J. (2) 70 (2018), no. 4, 511-521.
    
    
      \bibitem{futaki}
    {\scshape A. Futaki},
        \emph{An obstruction to the existence of Einstein K\"ahler metrics}, Invent. Math. {\bf 73} (1983), no. 3, 437-443.

      \bibitem{golota}
    {\scshape A. Golota},
        \emph{Delta-invariants for Fano varieties with large automorphism groups}, To appear in Int. J. Math.


   \bibitem{HK}
   {\scshape Y. Hu and S. Keel},
    \emph{ Mori dream spaces and GIT.} Michigan Math. J. Volume 48, Issue 1 (2000), 331-348.
    
    
        \bibitem{KKL}
   {\scshape A-S. Kaloghiros, A. K\"uronya and V. Lazi\'c},
    \emph{Finite generation and geography of models. } (Kyoto, 2011), 215–245,
Adv. Stud. Pure Math., 70, Math. Soc. Japan, [Tokyo], 2016. 


   \bibitem{singularitiesofpairs} 
{\scshape J. Koll{\'a}r}, \emph{ Singularities of pairs.} Algebraic geometry--Santa Cruz 1995, 221-287,
Proc. Sympos. Pure Math., 62, Part 1, Amer. Math. Soc., Providence, RI, 1997.


  \bibitem{Lazarsfeld1} 
{\scshape R. Lazarsfeld}, \emph{Positivity in algebraic geometry I: Classical setting line bundles and linear series}. Ergebnisse der Mathematik und ihrer Grenzgebiete. 3. Folge. A Series of Modern Surveys in Mathematics, 48. Springer-Verlag, Berlin, 2004. xviii+387 pp.

    \bibitem{li}
   {\scshape C. Li},
    \emph{K-semistability is equivariant volume minimization.} Duke Math. J. 166 (2017), no. 16, 3147-3218. 
    
        \bibitem{li-xu}
   {\scshape C. Li and C. Xu},
    \emph{Special test configuration and K-stability of Fano varieties.} Ann. of Math. (2) 180 (2014), no. 1, 197-232.
    
    
    \bibitem{LXZ}
       {\scshape Y. Liu, C. Xu and Z. Zhuang},
       \emph{Finite generation for valuations computing stability thresholds and applications to K-stability} 	arXiv:2102.09405,  36pp.

            \bibitem{odaka-calabi}
   {\scshape Y. Odaka},
    \emph{The Calabi conjecture and K-stability.} Int. Math. Res. Not. IMRN 2012, no. 10, 2272-2288.  

        
        \bibitem{odaka}
   {\scshape Y. Odaka},
    \emph{A generalization of the Ross-Thomas slope theory.} Osaka J. Math. 50 (2013), no. 1, 171-185. 
    
            \bibitem{phong-sturm}
   {\scshape D. Phong and J. Sturm},
    \emph{Test configurations for K-stability and geodesic rays.} J. Symplectic Geom. 5 (2007), no. 2, 221-247. 

    
        \bibitem{ross-thomas}
   {\scshape J. Ross and R. Thomas},
    \emph{An obstruction to the existence of constant scalar curvature K\"ahler metrics.} J. Differential Geom. 72 (2006), no. 3, 429-466. 

   
           \bibitem{szekelyhidi}
   {\scshape G. Sz\'ekelyhidi},
    \emph{On blowing up extremal K\"ahler manifolds.} Duke Math. J. 161 (2012), no. 8, 1411-1453. 
    
 \bibitem{szekelyhidi-filtrations}
    {\scshape G. Sz\'ekelyhidi},
    \emph{Filtrations and test-configurations.} With an appendix by Sebastien Boucksom.
Math. Ann. 362 (2015), no. 1-2, 451-484. 


    

            \bibitem{tian-inventiones}
   {\scshape G. Tian},
    \emph{K\"ahler-Einstein metrics with positive scalar curvature.} Invent. Math. 130 (1997), no. 1, 1-37.
   

    
            \bibitem{tian-book}
   {\scshape G. Tian},
    \emph{Canonical metrics in Kähler geometry.} Lectures in Mathematics ETH Zürich. Birkhäuser Verlag, Basel, 2000. vi+101 pp.
    
    \bibitem{wang}
   {\scshape X. Wang},
    \emph{Height and GIT weight} Math. Res. Lett. 19 (2012), no. 4, 909-926. 
    
 \bibitem{witt-nystrom}
   {\scshape D. Witt Nystr\"om},
    \emph{Test configurations and Okounkov bodies.} Compos. Math. 148 (2012), no. 6, 1736-1756. 
    
    \bibitem{yau}
       {\scshape S. T. Yau},
    \emph{Open problems in geometry.} Differential geometry: partial differential equations onmanifolds (Los Angeles, CA, 1990), Proc. Sympos. Pure Math., vol. 54, Amer. Math. Soc.,Providence, RI, 1993, pp. 1-28.
  
      \bibitem{zhang}
   {\scshape K. Zhang},
   \emph{Continuity of delta invariants and twisted K\"ahler--Einstein metrics}.  To appear in Adv. Math.
   
         \bibitem{KZ2}
   {\scshape K. Zhang},
   \emph{   A quantization proof of the uniform Yau-Tian-Donaldson conjecture}.  arXiv:2102.02438, 10pp.


       
      \bibitem{zhu}
   {\scshape Z. Zhu},
   \emph{A Note on Equivariant K-stability}.  To appear in Eur. J. Math.



\end{thebibliography}

\end{document}